\documentclass{amsart}

\usepackage{amsfonts,amsmath,tikz,xcolor,mathtools,extarrows, amsthm, oldgerm}
\usepackage[margin=1in]{geometry}
\usepackage[colorlinks=true,linkcolor=cyan,citecolor=cyan,urlcolor=cyan,backref]{hyperref}

\colorlet{darkblue}{blue!70!black}
\colorlet{darkred}{red!70!black}
\colorlet{darkgreen}{green!70!black}
\colorlet{darkwhite}{white!65!black}

\colorlet{darkmagenta}{magenta!70!black}
\colorlet{pink}{green!20!magenta}

\colorlet{lightcyan}{cyan!15!white}
\colorlet{lightyellow}{yellow!30!white}
\colorlet{lightcyanb}{lightcyan!50!black}

\newenvironment{tric}
    {\begin{tikzpicture}[scale= 0.5,line width=0.35mm,draw=darkblue,double distance=0.35mm,
        baseline={([yshift=-.8ex]current bounding box.center)}] }
    {\end{tikzpicture}}

\newtheorem{thm}{Theorem}[section]
\newtheorem{lemma}[thm]{Lemma}
\newtheorem{proposition}[thm]{Proposition}

\newtheorem{notation}[thm]{Notation}

\theoremstyle{definition}
\newtheorem{definition}[thm]{Definition}
\newtheorem{defn}[thm]{Definition}

\theoremstyle{remark}
\newtheorem{remark}[thm]{Remark}

\def\C{{\mathbb C}}

\newcommand{\FRep}{\mathbf{FundRep}}

\newcommand{\Web}{\mathbf{Web}_q}

\newcommand{\spn}[1][2n]{\mathfrak{sp}_{#1}}

\newcommand{\End}{{\rm End}}
\newcommand{\klabel}{k}
\newcommand{\PowerSum}{Q_{2n}^{(N,k)}}

\title{Type $C$ skein modules and transparent elements}
\date{}
\author[Higgins]{Vijay Higgins}
\address{Department of Mathematics, University of California, Los Angeles, CA 90095, USA}
\email{higginsv@math.ucla.edu}

\author[Wu]{Haihan Wu}
\address{Department of Mathematics, Johns Hopkins University, Baltimore, MD 21218 USA}
\email{hwu125@jhu.edu}

\thanks{
	2020 {\em Mathematics Classification:} Primary 57K31. Secondary 17B37.}

\begin{document}

\begin{abstract}
We study skein modules using $Sp(2n)$ webs. We define multivariable analogues of Chebyshev polynomials in the Type $C$ setting and use them to construct transparent elements in the skein module at roots of unity. Our arguments are diagrammatic and make use of an explicit braiding formula for $1$ and $k$ labeled strands and an analogue of Kuperberg's tetravalent vertex in the annular setting.
\end{abstract}

\maketitle

\tableofcontents

\section{Introduction}

\def\Braid
{\begin{tric}
\draw (-1.5,-1.5)--(1.5,1.5);
\draw[double=darkblue,ultra thick,white,line width=3pt] (-1.5,1.5)--(1.5,-1.5);
\draw (-1.5,-1.5) node[below,black,scale=0.7]{$1$};
\draw (1.5,-1.5) node[below,black,scale=0.7]{$k$};
\end{tric}}

\def\Braida
{\begin{tric}
\draw[scale=0.7] (-1.732,-2)--(0,-1)--(1.732,-2) (1.732,2)--(0,1)--(-1.732,2) 
(-1.732,-2)node[below,black,scale=0.7]{$1$}
(1.732,-2)node[below,black,scale=0.7]{$k$}
(1.732,2)node[above,black,scale=0.7]{$1$}
(-1.732,2) node[above,black,scale=0.7]{$k$}
(0,-1)--(0,1)node[left,midway,black,scale=0.7]{$k+1$};
\end{tric}
}

\def\Braidb
{\begin{tric}
\draw[scale=0.7] (-2,-1.732)--(-1,0)--(-2,1.732) (2,1.732)--(1,0)--(2,-1.732) 
(-2,-1.732)node[below,black,scale=0.7]{$1$}
(-2,1.732)node[above,black,scale=0.7]{$k$}
(2,1.732)node[above,black,scale=0.7]{$1$}
(2,-1.732) node[below,black,scale=0.7]{$k$}
(-1,0)--(1,0)node[below,midway,black,scale=0.7]{$k-1$};
\end{tric}
}

\def\Braidc
{\begin{tric}
\draw[scale=0.7] (-1.732,-2)--(0,-1)--(1.732,-2) (1.732,2)--(0,1)--(-1.732,2) 
(-1.732,-2)node[below,black,scale=0.7]{$1$}
(1.732,-2)node[below,black,scale=0.7]{$k$}
(1.732,2)node[above,black,scale=0.7]{$1$}
(-1.732,2) node[above,black,scale=0.7]{$k$}
(0,-1)--(0,1)node[left,midway,black,scale=0.7]{$k-1$};
\end{tric}
}

Associated to each Dynkin diagram is a quantum group $U_q(\mathfrak{g}).$ The category of representations of $U_q(\mathfrak{g})$ admits a diagrammatic description, as explained by Reshetikhin and Turaev \cite{RT90}. These diagrammatic descriptions have been explored to varying levels of detail, depending on $\mathfrak{g}.$ In \cite{Kup94,Kup96}, Kuperberg initiated the program to provide explicit combinatorial descriptions of representation categories by using diagrams called webs, which are graphs such that each edge carries the label of a fundamental representation of $U_q(\mathfrak{g}).$

Webs and their skein relations are powerful objects when viewed either as topological or algebraic tools. They contain enough topological data to define quantum invariants of links and 3-manifolds  while containing enough algebraic data to recover the representation category of the corresponding quantum group. Consequently, the combinatorial web categories provide for powerful constructions which are intermediate between the representation category of $U_q(\mathfrak{g})$ and the topological invariants that they define. For example, the web category constitutes the underlying 1-category inside of the 2-category of foams \cite{LQR15,QR16}. For another example, and which is the focus of this paper, the skein algebra of webs on a surface associates an algebra to a surface which is related to both the character variety of the surface and the corresponding quantum group itself.

For the case $\mathfrak{g}=\mathfrak{sl}_2,$ there is only one fundamental representation of $U_q(\mathfrak{sl}_2)$ and the webs do not require vertices. The corresponding combinatorial category is the ubiquitous Temperley-Lieb category. The problem of describing the combinatorial web category for each $\mathfrak{g}$ is an ongoing and active program. Kuperberg described the webs for each $\mathfrak{g}=\mathfrak{sl}_3,\mathfrak{sp}_4,\mathfrak{g}_2$. In type $A,$ the webs were described by \cite{MOY98,Sik05,Mor07,CKM14}. Webs for type $C$ were developed by Bodish-Elias-Rose-Tatham in \cite{BERT21}, and are the subject of this paper. In the case of type $C_2,$ skein algebras and cluster algebras were studied by Ishibashi-Yuasa in \cite{IY25}. The Type $C$ webs have been used by Kenyon-Wu in \cite{KW24} to study dimer models, generalizing a study of Douglas-Kenyon-Shi in the Type $A$ case \cite{DKS24}. Webs for the orthogonal group were described by Bodish-Wu in \cite{BW23}, which marks progress towards webs for type B and D. Progress has also been made for type $F_4$ \cite{SW24}.

For $\mathfrak{g}=\mathfrak{sl}_2,$ the corresponding link invariant is the Kauffman bracket, an unoriented framed version of the Jones polynomial. The skein algebra corresponding to $U_q(\mathfrak{sl}_2)$ is often referred to as the Kauffman bracket skein algebra. It has been a popular object to study in the last decade and several major results are known \cite{BW11,BW16,FKBL19,GJS19}. Skein algebras are generally noncommutative and their centers are generally trivial when $q$ is a generic paremeter. However, when $q$ is specialized to a root of unity, there is a rich center. The work of Bonahon-Wong \cite{BW16} showed that when $q$ is a root of unity of order $N,$ the central elements can be obtained by threading the Chebyshev polynomial $T_N$ along knots. The work of Frohman-Kania-Bartoszynska-L\^{e} \cite{FKBL19} shows that the whole center can be obtained in this way. We expect the polynomials we construct in this paper will play a similar role for the type $C$ case.

From another perspective, when $q$ is generic, the Chebyshev skeins associated to both $T_n$ and the closely related polynomials $S_n$ arise in the study of positive bases for skein algebras \cite{Thur14,Queff22,MQ23}, which have relations to both canonical bases for cluster algebras and a categorification of skein modules of surfaces via foams.

For some cases other than $\mathfrak{g}=\mathfrak{sl}_2,$ analogues of the polynomials $T_n$ and $S_n$ and their corresponding Chebyshev elements have been studied. The analogues of Jones-Wenzl projecors for $\mathfrak{g}=\mathfrak{sl}_3,\mathfrak{sp}_4,\mathfrak{g}_2$ have been explicitly described in \cite{OY97,Kim07,Bod22,BW23b}. The annular closure of such projectors should produce the analogues of the polynomials $S_n$ for these $\mathfrak{g}.$ Analogues of the polynomials $T_n$ are referred to as power-sum polynomials, and have been used in \cite{BH24,BGBHHMP23} to construct central elements in skein algebras for the case of $SL_n$ and $G_2.$ Analogues of power-sum polynomials also arise in the study of HOMFLY-PT and Kauffman skein algebras \cite{MS17,MPS23}.

The goal of this paper is to describe the $Sp_{2n}$ analogues of the power sum elementary polynomials previously described for the case of $SL_n$ in  \cite{BH24}, and then to use them to construct central elements in $Sp_{2n}$ skein algebras at roots of unity.

\subsection{Main results}

We construct a polynomial in $n$ variables $Q_{2n}^{(N,k)} \in \mathbb{Z}[f_1,\dots,f_n]$ for each $k=1,\dots n.$ We then show that when $q$ is a suitable root of unity, each polynomial can be used to construct transparent elements of the Type $C_n$ skein module by way of threading the polynomial along a framed knot.

\begin{thm}(Theorem \ref{main theorem})
When $q^{2N}=1,$ the element obtained by threading $Q_{2n}^{(N,k)}$ along any framed knot $K$ in the skein module $\mathcal{S}_q^{C_n}(M)$ of an oriented $3\text{-manifold } M$ is a transparent element of the skein module. When the module $M$ is a thickened surface $\Sigma \times (0,1)$ the element constructed in this way is a central element of the skein algebra of $\Sigma.$
\end{thm}

We call the polynomials $Q_{2n}^{(N,k)}$ power fundamental polynomials, and they are Type $C$ analogues of the power elementary polynomials $\hat{P}_{2n}^{(N,k)}$ for Type $A$ constructed in \cite{BH24}. Instead of constructing the polynomials $Q_{2n}^{(N,k)}$ from scratch, we obtain them by taking a certain specialization of a linear combination of the Type $A$ polynomials $\hat{P}_{2n}^{(N,k)}$. Our technique of obtaining Type $C_n$ polynomials from Type $A_{2n}$ might be related to the technique of folding, which has recently been successful in \cite{BER24} to obtain Type $B_n$ link homology from Type $A_{2n-1}.$

The proof of our main theorem is diagrammatic and relies on computations with Type $C$ webs in an annulus. In \cite{BERT21}, a formula giving the braiding of two strands labeled by $1$ was given. Our computations required us to braid strands labeled by $1$ and $k$ for any $1 \leq k \leq n.$ To this end we obtained the following closed formula.

\begin{thm}(Theorem \ref{Braiding theorem})
In the Type $C$ skein algebra, the braiding formula for strands labeled by $1$ and $k$ is given by
\begin{equation*}
       (-1)^k \Braid :=   \Braida - q^{-1} \Braidb  - \frac{q^{n-k+1}}{[n-k+1]} \Braidc
   \end{equation*}
\end{thm}

To align the combinatorics of the skein theory with that of the Type $C$ power fundamental polynomials, we employ the use of a certain family of elements $X_k$ in the annulus which arise from an annular generalization of the $4\text{-valent}$ vertex used by Kuperberg. Kuperberg famously introduced a $4\text{-valent}$ vertex to construct bases of $Sp_4$ web spaces in \cite{Kup96}. An $Sp_{2n}$ analogue of Kuperberg's $4\text{-valent}$  vertex was used in \cite{BERT21} to construct bases of $Sp_{2n}$ web spaces. Our element $X_1$ in the annulus comes directly from the vertex used in \cite{BERT21} while our elements $X_k$ for $k \geq 2$ are defined in the annular setting only and thus represent a new kind of generalization. 

\subsection{Acknowledgments.}
V.H. was partially supported by National Science Foundation grant DMS-$2136090$. H.W. was supported by National Science Foundation grant CCF-$2317280$ and Simons Foundation grant $994328$. We would like to thank Elijah Bodish and Monica Vazirani for helpful conversations.

\section{Type C web category and skein algebra}

\subsection{Type C webs}

We now recall the definition of the Type $C$ web category from \cite{BERT21}, denoted by $\textbf{Web}_q(\mathfrak{sp}_{2n}).$

\def\Vertc
{\begin{tric}
\draw [scale=0.8] (0,0)--(90:1) (0,0)--(210:1) (0,0)--(330:1);
\draw (90:1)node[black,anchor=south,scale=0.7]{$k+1$}
      (210:1)node[black,anchor=north,scale=0.7]{$1$}
      (330:1)node[black,anchor=north,scale=0.7]{$k$}; 
\end{tric}
}

\def\Vertd
{\begin{tric}
\draw [scale=0.8] (0,0)--(90:1) (0,0)--(210:1) (0,0)--(330:1);
\draw (90:1)node[black,anchor=south,scale=0.7]{$k+1$}
      (210:1)node[black,anchor=north,scale=0.7]{$k$}
      (330:1)node[black,anchor=north,scale=0.7]{$1$}; 
\end{tric}
}

\def\Skeinak
{\begin{tric}
\draw (0.7,0) circle (0.7);
\draw (0,0)node[left,black,scale=0.7]{$k$}; 
\end{tric}
}

\def\Skeinaone
{\begin{tric}
\draw (0.7,0) circle (0.7);
\draw (0,0)node[left,black,scale=0.7]{$1$}; 
\end{tric}
}

\def\Skeingg
{\begin{tric}
\draw[scale=0.8] (0,-1)--(0,0) (0,0)..controls(0.7,0.5)and(0.7,1.5)..(0,1.5)..controls(-0.7,1.5)and(-0.7,0.5)..(0,0);
\draw (0,-0.8)node[below,black,scale=0.7]{$2$};
\draw (-0.4,0.7)node[left,black,scale=0.7]{$1$};
\end{tric}
}

\def\Skeinbigon
{\begin{tric}
\draw[scale= 0.8] (0,0.7)--(0,1.5) (0,-0.7)--(0,-1.5)
        (0,1.5)node[above,black,scale=0.7]{$k$}
        (0,-1.5)node[below,black,scale=0.7]{$k$}
      (0,0.7)..controls(-0.5,0.7)and(-0.5,-0.7)..(0,-0.7) node[left,midway,black,scale=0.7]{$1$} (0,0.7)..controls(0.5,0.7)and(0.5,-0.7)..(0,-0.7)
      node[right,midway,black,scale=0.7]{$k-1$} ;
\end{tric}
}

\def\Skeinbigona
{\begin{tric}
\draw [scale=0.8] (0,1.5)--(0,-1.5)  node[below,black,scale=0.7]{$k$} ;
\end{tric}
}

\def\FlowSkeinbigon
{\begin{tric}
\draw[scale= 0.8] (0,0.7)--(0,1.5) (0,-0.7)--(0,-1.5)
        (0,1.5)node[above,black,scale=0.7]{$k+1$}
        (0,-1.5)node[below,black,scale=0.7]{$k-1$}
      (0,0.7)..controls(-0.5,0.7)and(-0.5,-0.7)..(0,-0.7) node[left,midway,black,scale=0.7]{$1$} (0,0.7)..controls(0.5,0.7)and(0.5,-0.7)..(0,-0.7)
      node[right,midway,black,scale=0.7]{$k$} ;
\end{tric}
}

\def\SkeinIH
{\begin{tric}
\draw[scale=0.7] (-2,-1.732)--(-1,0)--(-2,1.732) (2,1.732)--(1,0)--(2,-1.732) 
(-2,-1.732)node[below,black,scale=0.7]{$k$}
(-2,1.732)node[above,black,scale=0.7]{$1$}
(2,1.732)node[above,black,scale=0.7]{$1$}
(2,-1.732) node[below,black,scale=0.7]{$k$}
(-1,0)--(1,0)node[below,midway,black,scale=0.7]{$k+1$};
\end{tric}
}

\def\GrSkeinIH
{\begin{tric}
\draw[darkgreen,scale=0.7] (-2,-1.732)--(-1,0)--(-2,1.732) (2,1.732)--(1,0)--(2,-1.732) 
(-2,-1.732)node[below,black,scale=0.7]{$k$}
(-2,1.732)node[above,black,scale=0.7]{$1$}
(2,1.732)node[above,black,scale=0.7]{$1$}
(2,-1.732) node[below,black,scale=0.7]{$k$}
(-1,0)--(1,0)node[below,midway,black,scale=0.7]{$k+1$};
\end{tric}
}

\def\SkeinIHa
{\begin{tric}
\draw [scale=0.7]  (-1.74,2)--(0,1)--(1.74,2) 
(1.74,-3)--(0,0) node[right,pos=0.7,black,scale=0.7]{$1$}
(0,0)--(-1.74,-3) node[left,pos=0.3,black,scale=0.7]{$1$}
(0,0)--(0,1) node[left,midway,black,scale=0.7]{$2$}
(-1.74,-3) node[below,black,scale=0.7]{$k$}
(1.74,-3) node[below,black,scale=0.7]{$k$}
(-1.74,2) node[above,black,scale=0.7]{$1$}
(1.74,2) node[above,black,scale=0.7]{$1$}
(1.16,-2)--(-1.16,-2)node[below,midway,black,scale=0.7]{$k-1$};
\end{tric}
}

\def\GrSkeinIHa
{\begin{tric}
\draw [darkgreen, scale=0.7]  (-1.74,2)--(0,1)--(1.74,2) 
(1.74,-3)--(0,0) node[right,pos=0.7,black,scale=0.7]{$1$}
(0,0)--(-1.74,-3) node[left,pos=0.3,black,scale=0.7]{$1$}
(0,0)--(0,1) node[left,midway,black,scale=0.7]{$2$}
(-1.74,-3) node[below,black,scale=0.7]{$k$}
(1.74,-3) node[below,black,scale=0.7]{$k$}
(-1.74,2) node[above,black,scale=0.7]{$1$}
(1.74,2) node[above,black,scale=0.7]{$1$}
(1.16,-2)--(-1.16,-2)node[below,midway,black,scale=0.7]{$k-1$};
\end{tric}
}

\def\SkeinIHb
{\begin{tric}
\draw[scale=0.7] (-2,-1.732)--(-1,0)--(-2,1.732) (2,1.732)--(1,0)--(2,-1.732) 
(-2,-1.732)node[below,black,scale=0.7]{$k$}
(-2,1.732)node[above,black,scale=0.7]{$1$}
(2,1.732)node[above,black,scale=0.7]{$1$}
(2,-1.732) node[below,black,scale=0.7]{$k$}
(-1,0)--(1,0)node[below,midway,black,scale=0.7]{$k-1$};
\end{tric}
}

\def\GrSkeinIHb
{\begin{tric}
\draw[darkgreen,scale=0.7] (-2,-1.732)--(-1,0)--(-2,1.732) (2,1.732)--(1,0)--(2,-1.732) 
(-2,-1.732)node[below,black,scale=0.7]{$k$}
(-2,1.732)node[above,black,scale=0.7]{$1$}
(2,1.732)node[above,black,scale=0.7]{$1$}
(2,-1.732) node[below,black,scale=0.7]{$k$}
(-1,0)--(1,0)node[below,midway,black,scale=0.7]{$k-1$};
\end{tric}
}

\def\SkeinIHc
{\begin{tric}
\draw  [scale=0.7] 
(-2,2)..controls(-1,1)and(1,1)..(2,2) node[above,black,midway,scale=0.7]{$1$}
(2,-2)..controls(1,-1)and(-1,-1)..(-2,-2) node[below,black,midway,scale=0.7]{$k$};
\end{tric}
}

\def\GrSkeinIHc
{\begin{tric}
\draw  [darkgreen,scale=0.7] 
(-2,2)..controls(-1,1)and(1,1)..(2,2) node[above,black,midway,scale=0.7]{$1$}
(2,-2)..controls(1,-1)and(-1,-1)..(-2,-2) node[below,black,midway,scale=0.7]{$k$};
\end{tric}
}

\def\Skeinassocia
{\begin{tric}
\draw[scale=0.8] (0,0)..controls(0,0.5)and(0.2,0.7)..(0.5,1)
      (1,0)..controls(1,0.5)and(0.8,0.7)..(0.5,1)
      (0.5,1)..controls(0.5,1.5)and(0.3,1.7)..(0,2) 
              node[right,black,midway, scale=0.7]{$k+1$}
      (-1,0)..controls(-1,1)and(-0.6,1.5)..(0,2)
      (0,2)--(0,3)
      (0,0)node[below,black,scale=0.7]{$k$}
      (1,0)node[below,black,scale=0.7]{$1$}
      (-1,0)node[below,black,scale=0.7]{$1$}
      (0,3)node[above,black,scale=0.7]{$k+2$};
\end{tric}
}

\def\Skeinassoci
{\begin{tric}
\draw[scale=0.8] (0,0)..controls(0,0.5)and(-0.2,0.7)..(-0.5,1)
      (-1,0)..controls(-1,0.5)and(-0.8,0.7)..(-0.5,1)
      (-0.5,1)..controls(-0.5,1.5)and(-0.3,1.7)..(0,2) 
              node[left,black,midway, scale=0.7]{$k+1$}
      (1,0)..controls(1,1)and(0.6,1.5)..(0,2)
      (0,2)--(0,3)
      (0,0)node[below,black,scale=0.7]{$k$}
      (-1,0)node[below,black,scale=0.7]{$1$}
      (1,0)node[below,black,scale=0.7]{$1$}
      (0,3)node[above,black,scale=0.7]{$k+2$};
\end{tric}
}

Let $\mathcal{R}$ denote a commutative unital integral domain containing an invertible element $q \in \mathcal{R}$ such that $[k]^{-1} \in \mathcal{R},$ for $k= 1,2,\dots ,n,$ where the quantum integer $[n]$ is the Laurent polynomial expressed by $[n]:= \dfrac{q^n-q^{-n}}{q-q^{-1}}$. 
Denote by $ [n]! :=  [1] [2] [3] ... [n]$.

\begin{defn}[{\cite[Definition 1.1]{BERT21}}] 
 The pivotal $\mathcal{R}$-linear category $\textbf{Web}_q(\mathfrak{sp}_{2n}) $ is the category whose objects are generated by the self-dual objects $ k \in \{1,2,...,n \}$, and whose morphism spaces are spanned by $C_n$ webs with internal vertices of the following forms:
 \begin{align*}
     \Vertc \in {\text{Hom}}_{\textbf{Web}_q(\mathfrak{sp}_{2n})}(1\otimes k, k+1)\ \ ,  \ \ \ \ \ \ \ \ \ \ \ \ \ \  
     \Vertd \in {\text{Hom}}_{\textbf{Web}_q(\mathfrak{sp}_{2n})}(k \otimes 1, k+1) \ 
 \end{align*}
for $0 \leq k \leq n-1$, modulo the following relations with the conventions that any strand carrying the label $0$ may be erased and a diagram is zero if it contains a strand of label $l$ with either $l>n$ or $l<0.$

\begin{align} 
  \label{defskein}
&\text{(\ref{defskein}a)} \Skeinaone = -\frac{[n][2n+2]}{[n{+}1]},  \ \ 
\text{(\ref{defskein}b)}\Skeingg=\ 0 \   ,\ \ 
\text{(\ref{defskein}c)}\Skeinbigon  \ = \ [k] \ \Skeinbigona \ \  , \ \ 
\text{(\ref{defskein}d)}\Skeinassoci =\Skeinassocia  
\\
& \text{(\ref{defskein}e)} \SkeinIH =  \SkeinIHa  
-\frac{[n{-}k]}{[n{-}k{+}1]}\SkeinIHb
  +  \frac{[n{-}k]}{[n]}  \SkeinIHc. \notag
\end{align}
\end{defn}

\def\IndBraida
{\begin{tric}
\draw[scale=0.7] (-1.732,-2)--(0,-1)--(1.732,-2) (1.732,2)--(0,1)--(-1.732,2) 
(-1.732,-2)node[below,black,scale=0.7]{$1$}
(1.732,-2)node[below,black,scale=0.7]{$k+1$}
(1.732,2)node[above,black,scale=0.7]{$1$}
(-1.732,2) node[above,black,scale=0.7]{$k+1$}
(0,-1)--(0,1)node[left,midway,black,scale=0.7]{$k+2$};
\end{tric}
}

\def\IndBraidc
{\begin{tric}
\draw[scale=0.7] (-1.732,-2)--(0,-1)--(1.732,-2) (1.732,2)--(0,1)--(-1.732,2) 
(-1.732,-2)node[below,black,scale=0.7]{$1$}
(1.732,-2)node[below,black,scale=0.7]{$k+1$}
(1.732,2)node[above,black,scale=0.7]{$1$}
(-1.732,2) node[above,black,scale=0.7]{$k+1$}
(0,-1)--(0,1)node[left,midway,black,scale=0.7]{$k$};
\end{tric}
}

\def\NBraida
{\begin{tric}
\draw[scale=0.5] (-1.732,-2)--(0,-1)--(1.732,-2) (1.732,2)--(0,1)--(-1.732,2) 
(-1.732,-2)node[below,black,scale=0.7]{$1$}
(1.732,-2)node[below,black,scale=0.7]{$n$}
(1.732,2)node[above,black,scale=0.7]{$1$}
(-1.732,2) node[above,black,scale=0.7]{$n$}
(0,-1)--(0,1)node[left,midway,black,scale=0.7]{$n+1$};
\end{tric}
}

\def\IndBraidb
{\begin{tric}
\draw[scale=0.7] (-2,-1.732)--(-1,0)--(-2,1.732) (2,1.732)--(1,0)--(2,-1.732) 
(-2,-1.732)node[below,black,scale=0.7]{$1$}
(-2,1.732)node[above,black,scale=0.7]{$k+1$}
(2,1.732)node[above,black,scale=0.7]{$1$}
(2,-1.732) node[below,black,scale=0.7]{$k+1$}
(-1,0)--(1,0)node[below,midway,black,scale=0.7]{$k$};
\end{tric}
}

\def\IndBraidbH
{\begin{tric}
\draw[scale=0.7] (-2,-1.732)--(-1,0)--(-2,1.732) (2,1.732)--(1,0)--(2,-1.732) 
(-2,-1.732)node[below,black,scale=0.7]{$1$}
(-2,1.732)node[above,black,scale=0.7]{$k+1$}
(2,1.732)node[above,black,scale=0.7]{$1$}
(2,-1.732) node[below,black,scale=0.7]{$k+1$}
(-1,0)--(1,0)node[below,midway,black,scale=0.7]{$k+2$};
\end{tric}
}

\def\Braidd
{\begin{tric}
\draw[scale=0.7] (-2,-1.732)--(-1,0)--(-2,1.732) (2,1.732)--(1,0)--(2,-1.732) 
(-2,-1.732)node[below,black,scale=0.7]{$1$}
(-2,1.732)node[above,black,scale=0.7]{$k$}
(2,1.732)node[above,black,scale=0.7]{$1$}
(2,-1.732) node[below,black,scale=0.7]{$k$}
(-1,0)--(1,0)node[below,midway,black,scale=0.7]{$k+1$};
\end{tric}
}

We note that the expression on the right side of Equation (\ref{defskein}a) is to be interpreted as the Laurent polynomial in generic $q$ before specialization to our $q \in \mathcal{R}$ and so is defined even if $[n+1]$ is not invertible in $\mathcal{R}.$

The category $\FRep(U_q(\spn))$ is the full subcategory of the representation category of the quantum group $U_q(\spn)$ whose objects are tensor powers of fundamental representations of $U_q(\spn).$ The main theorem in \cite[Theorem 1.4]{BERT21} shows that when $\mathcal{R}=\mathbb{C}(q),$ there is an equivalence of $\C(q)$-linear ribbon categories 
\[
\Phi \colon \Web(\spn) \to \FRep(U_q(\spn))
\]
sending $k$ to $V_{\varpi_k},$ the $k\text{-th}$ fundamental representation of $U_q(\spn).$ We will discuss the braiding for Type $C$ webs in Section \ref{braiding formulas section}.

The following skein relation follows from the defining relations. It will play a role in our construction.

\begin{proposition}[{\cite[Theorem 3.6]{BERT21}}] 

\begin{equation} \label{deducedskein}
    \Braida + \frac{[n-k]}{[n-k+1]} \Braidc = \Braidd + \frac{[n-k]}{[n-k+1]} \Braidb
\end{equation}

\end{proposition}

\subsection{Braiding formulas}\label{braiding formulas section}

The following skein relation gives the definition of the braiding of two strands, both labeled by $1.$

\def\BraidOneOne
{\begin{tric}
\draw (1.5,-1.5)--(-1.5,1.5);
\draw[double=darkblue,ultra thick,white,line width=3pt] (1.5,1.5)--(-1.5,-1.5);
\draw (-1.5,-1.5) node[below,black,scale=0.7]{$1$};
\draw (1.5,-1.5) node[below,black,scale=0.7]{$1$};
\end{tric}}

\def\BraidOneOneInv
{\begin{tric}
\draw (-1.5,-1.5)--(1.5,1.5);
\draw[double=darkblue,ultra thick,white,line width=3pt] (-1.5,1.5)--(1.5,-1.5);
\draw (-1.5,-1.5) node[below,black,scale=0.7]{$1$};
\draw (1.5,-1.5) node[below,black,scale=0.7]{$1$};
\end{tric}}

\def\OneOneTwo
{\begin{tric}
\draw[scale=0.7] (-1.732,-2)--(0,-1)--(1.732,-2) (1.732,2)--(0,1)--(-1.732,2) 
(-1.732,-2)node[below,black,scale=0.7]{$1$}
(1.732,-2)node[below,black,scale=0.7]{$1$}
(1.732,2)node[above,black,scale=0.7]{$1$}
(-1.732,2) node[above,black,scale=0.7]{$1$}
(0,-1)--(0,1)node[left,midway,black,scale=0.7]{$2$};
\end{tric}
}

\def\IdWeb
{\begin{tric}
\draw   
(-1.5,1.5)..controls(-0.7,0.7)and(-0.7,-0.7)..(-1.5,-1.5) 
(1.5,-1.5)..controls(0.7,-0.7)and(0.7,0.7)..(1.5,1.5) ;
\draw (-1.5,-1.5) node[below,black,scale=0.7]{$1$};
\draw (1.5,-1.5) node[below,black,scale=0.7]{$1$};
\end{tric}
}

\def\CupCap
{\begin{tric}
\draw 
(-1.5,1.5)..controls(-0.7,0.7)and(0.7,0.7)..(1.5,1.5) 
(1.5,-1.5)..controls(0.7,-0.7)and(-0.7,-0.7)..(-1.5,-1.5) ;
\draw (-1.5,-1.5) node[below,black,scale=0.7]{$1$};
\draw (1.5,-1.5) node[below,black,scale=0.7]{$1$};
\end{tric}
}

\begin{proposition} 
[{\cite[Porism 5.4]{BERT21}}] 
The elements
\begin{equation} \label{InitialBraid}
\BraidOneOne=q\IdWeb +\cfrac{q^{-n}}{[n]} \OneOneTwo - \CupCap  
\end{equation}
and 
\begin{equation*}
\BraidOneOneInv= q^{-1} \IdWeb + \cfrac{q^n}{[n]}\OneOneTwo - \CupCap 
\end{equation*}
in $\End_{\Web(\spn)}(1 \otimes 1)$ are inverse isomorphisms that are sent by the functor $\Phi$ to the 
braiding $\beta_{V_1,V_1}$ and its inverse.
\end{proposition}

\def\Braidre
{\begin{tric}
\draw (1.5,-1.5)--(-1.5,1.5);
\draw[double=darkblue,ultra thick,white,line width=3pt] (-1.5,-1.5)--(1.5,1.5);
\draw (-1.5,-1.5) node[below,black,scale=0.7]{$1$};
\draw (1.5,-1.5) node[below,black,scale=0.7]{$k$};
\end{tric}}

\def\BraidP
{\begin{tric}
\draw (-1.5,-1.5)--(1.5,1.5);
\draw[double=darkblue,ultra thick,white,line width=3pt] (-1.5,1.5)--(1.5,-1.5);
\draw (-1.5,-1.5) node[below,black,scale=0.7]{$1$};
\draw (1.5,-1.5) node[below,black,scale=0.7]{$k+1$};
\end{tric}}

\def\BraidPa
{\begin{tric}
\draw (-1.5,-1.5)--(1.5,1.5);
\draw[double=darkblue,ultra thick,white,line width=3pt] (-0.5,0.5)--(1.5,-1.5);
\draw (-1.5,-1.5) node[below,black,scale=0.7]{$1$};
\draw (1.5,-1.5) node[below,black,scale=0.7]{$k+1$};
\draw (-0.5,0.5)..controls(-1.5,0.5)..(-1.5,1.5) node[pos=0.3,below,black,scale=0.7]{$1$}
      (-0.5,0.5)..controls(-0.5,1.5)..(-1.5,1.5) node[midway,above,black,scale=0.7]{$k$}
      (-1.5,1.5)--(-2,2);  
\end{tric}}

\def\BraidPb
{\begin{tric}
\draw (-1.5,-1.5)--(1.5,1.5) (0.7,-0.7)--(1.5,-1.5) (-0.7,0.7)--(-1.5,1.5);
\draw[double=darkblue,ultra thick,white,line width=3pt] (-0.7,0.7)..controls(0.7,0.7)..(0.7,-0.7) node[pos=0.2,above,black,scale=0.7]{$k$}
(-0.7,0.7)..controls(-0.7,-0.7)..(0.7,-0.7)
node[pos=0.8,below,black,scale=0.7]{$1$};
\draw (-1.5,-1.5) node[below,black,scale=0.7]{$1$};
\draw (1.5,-1.5) node[below,black,scale=0.7]{$k+1$};
\end{tric}}

\def\BraidLK
{\begin{tric}
\draw (-1.5,-1.5)--(1.5,1.5);
\draw[double=darkblue,ultra thick,white,line width=3pt] (-1.5,1.5)--(1.5,-1.5);
\draw (-1.5,-1.5) node[below,black,scale=0.7]{$l+1$};
\draw (1.5,-1.5) node[below,black,scale=0.7]{$k+1$};
\end{tric}}

\def\BraidLKa
{\begin{tric}
\draw (-1.7,-1.7)--(-1.2,-1.2) (1.7,1.7)--(1.2,1.2) (1.7,-1.7)--(1.2,-1.2) (-1.7,1.7)--(-1.2,1.2);

\draw (1.2,1.2)..controls(-0.5,0.5)..(-1.2,-1.2) node[pos=0.15,above,black,scale=0.7]{$l$}
(1.2,1.2)..controls(0.5,-0.5)..(-1.2,-1.2)
node[pos=0.85,below,black,scale=0.7]{$1$};

\draw[double=darkblue,ultra thick,white,line width=3pt] (-1.2,1.2)..controls(0.5,0.5)..(1.2,-1.2) node[pos=0.2,above,black,scale=0.7]{$k$}
(-1.2,1.2)..controls(-0.5,-0.5)..(1.2,-1.2)
node[pos=0.85,below,black,scale=0.7]{$1$};

\draw (-1.7,-1.7) node[below,black,scale=0.7]{$l+1$};
\draw (1.7,-1.7) node[below,black,scale=0.7]{$k+1$};
\draw (-1.7,1.7) node[above,black,scale=0.7]{$k+1$};
\draw (1.7,1.7) node[above,black,scale=0.7]{$l+1$};
\end{tric}}

Consequently, the braiding of strands labeled by $k,l$ can be expressed by the following relation.

\begin{definition}[{\cite[Section 5.9]{BERT21}}]
 The braiding of two arbitrarily labeled strands can be defined inductively by the following formula.
 \begin{equation}\label{klbraiding}
 \BraidLK :=\frac{1}{[l+1][k+1]} \BraidLKa, \quad 0 \leq k,l \leq n-1.    
 \end{equation}
 
\end{definition}

 A simple closed formula for the braiding of strands labeled $k$ and $l$ analogous to those in \cite{MOY98,CKM14} appears to be an open problem and should involve the more general flow vertices with incident strands labeled by $a,b,a+b$ in \cite[Definition 3.2]{BERT21}. However, our results will only require such a formula in the special case that $l=1.$ We obtain the following formula, and provide its full proof in Appendix \ref{Braidproof}.

\begin{thm}\label{Braiding theorem}
       \begin{equation} \label{1kbrading}
       (-1)^k \Braid :=   \Braida - q^{-1} \Braidb  - \frac{q^{n-k+1}}{[n-k+1]} \Braidc
   \end{equation}
\end{thm}

\begin{proof}
See Appendix \ref{Braidproof}.
\end{proof}

\begin{proposition}
   \begin{equation*} 
       (-1)^k \Braidre :=   \Braida - q \Braidb  - \frac{q^{-n+k-1}}{[n-k+1]} \Braidc
   \end{equation*}
\end{proposition}

\begin{proof}
    The two braidings in equation (\ref{InitialBraid}) satisfy the property that one crossing is sent to the other by making the switch $q \leftrightarrow q^{-1}.$ Since the braiding in (\ref{klbraiding}) is defined inductively in terms of the braiding of $1\text{-labeled}$ strands in equation (\ref{InitialBraid}), the braiding obtained by switching all crossings in (\ref{klbraiding}) is obtained by making the switch $q \leftrightarrow q^{-1}.$
\end{proof}

\subsection{Type C skein algebra}

In this section, we provide the definition of the $C_n$ skein module of an oriented 3-manifold and $C_n$ skein algebra of a surface. Let $M$ be an oriented $3\text{-manifold}.$

\begin{definition}
A $C_n$ web $W$ in $M$ is an embedding of a closed $C_n$ web equipped with a ribbon structure, which is a thin oriented surface in $M$ which deformation retracts onto $W.$ By convention, we allow the empty web and webs without vertices, which are the same as framed links.
\end{definition}

\begin{definition}
We define the $C_n$ skein module of $M$ denoted by $\mathcal{S}_q^{C_n}(M)$ to be the module spanned over $\mathcal{R}$ by isotopy classes of $C_n$ webs in $M$ modulo the relations (\ref{defskein}a)-(\ref{defskein}e),  \eqref{InitialBraid}, and \eqref{klbraiding}.
\end{definition}

\begin{definition}
In the case that $M=\Sigma \times (0,1)$ is a thickened oriented surface, the skein module of $M$ carries an algebra structure called the skein algebra of $\Sigma,$ denoted by $\mathcal{S}_q^{C_n}(\Sigma)$. In the skein algebra, the product of two webs $[W_1][W_2]$ is given by isotoping $W_1$ so that it is contained in $\Sigma \times (1/2,1),$ isotoping $W_2$ so that is contained in $\Sigma \times (0,1/2)$ and then taking $W_1W_2=[W_1\cup W_2]\in \mathcal{S}_q^{C_n}(\Sigma).$ 
\end{definition}

When working with a skein algebra of a surface, we will often work with projections of the webs onto the surface and perform computations with such web diagrams. We note that the naturality of the braiding in $\Web(\spn)$ follows from that of $\FRep(U_q(\spn))$, and consequently the web relations imply the following types of Reidemeister moves hold for all strand labels:


\def\ReideA
{\begin{tric}
\draw  (-1.5,1.5)..controls(0,-1.5)and(1.5,-1.5)..(1.5,0);
\draw[double=darkblue,ultra thick,white,line width=3pt] (-1.5,-1.5)..controls(0,1.5)and(1.5,1.5)..(1.5,0);
\end{tric}
}

\def\ReideotherA
{\begin{tric}
\draw (-1.5,-1.5)..controls(0,1.5)and(1.5,1.5)..(1.5,0);
\draw [double=darkblue,ultra thick,white,line width=3pt] (-1.5,1.5)..controls(0,-1.5)and(1.5,-1.5)..(1.5,0);
\end{tric}
}

\def\ReideAa
{\begin{tric}
\draw  [scale=0.7] 
(-2,2)..controls(-1,1)and(-1,-1)..(-2,-2) node[left,black,midway,scale=0.7]{$1$}
(1,-2)..controls(-0.5,-2)and(-0.5,2)..(1,2) node[right,black,midway,scale=0.7]{$1$}
(1,-2)..controls(2.5,-2)and(2.5,2)..(1,2) ;
\end{tric}
}

\def\ReideAb
{\begin{tric}
\draw[scale=0.7] (-1.732,-2)--(0,-1)  (0,1)--(-1.732,2) 
(0,-1)..controls(2,-1)and(2,1)..(0,1) node[left,black,midway,scale=0.7]{$1$}
(-1.732,-2)node[below,black,scale=0.7]{$1$}
(-1.732,2) node[above,black,scale=0.7]{$1$}
(0,-1)--(0,1)node[left,midway,black,scale=0.7]{$2$};
\end{tric}
}

\def\ReideAc
{\begin{tric}
\draw  [scale=0.7] 
(-2,2)..controls(-0.5,1)..(1,1) node[above,black,midway,scale=0.7]{$1$}
(-2,-2)..controls(-0.5,-1)..(1,-1)
(1,1)..controls(2,1)and(2,-1)..(1,-1); 
\end{tric}
}

\def\ReideAhalf
{\begin{tric}
\draw[scale=0.7] (-1.5,-1.5)..controls(-1.5,-1)and(-1,-0.5)..(0,-0.5)
                 (1.5,-1.5)..controls(1.5,-1)and(1,-0.5).. (0,-0.5)
                 (0,-0.5)--(0,1)node[left,black,midway,scale=0.7]{$2$}; 

\draw[scale=0.7] (-1.5,-1.5)..controls(-1.5,-3)and(1.5,-3) ..(1.5,-4.5)
node[below,black,scale=0.7]{$1$};
\draw[scale=0.7, double=darkblue,ultra thick,white,line width=3pt] (1.5,-1.5)..controls(1.5,-3)and(-1.5,-3) ..(-1.5,-4.5)node[below,black,scale=0.7]{$1$};
\end{tric}
}

\def\ReidemCa
{\begin{tric}
\draw  (-2,2)--(2,-2);
\draw [double=darkblue,ultra thick,white,line width=3pt] (0,-2)..controls(-2,-1)and(-2,1)..(0,2);
\draw [double=darkblue,ultra thick,white,line width=3pt](-2,-2)--(2,2) ;
\end{tric}
}

\def\ReidemCb
{\begin{tric}
\draw  (-2,2)--(2,-2);
\draw [double=darkblue,ultra thick,white,line width=3pt] (0,-2)..controls(2,-1)and(2,1)..(0,2);
\draw [double=darkblue,ultra thick,white,line width=3pt](-2,-2)--(2,2) ;
\end{tric}
}

\def\ReidemDa
{\begin{tric}
\draw        (90:0)--(90:2) 
           (306:2)--(90:0)--(234:2); 
\draw [double=darkblue,ultra thick,white,line width=3pt]   (162:2)..controls(126:1.6)and(54:1.6)..(18:2);
\end{tric}
}

\def\ReidemDb
{\begin{tric}
\draw        (90:0)--(90:2) 
           (306:2)--(90:0)--(234:2) ;
\draw [double=darkblue,ultra thick,white,line width=3pt] (162:2)..controls(234:1.7)and(306:1.7)..(18:2);
\end{tric}
}

\def\ReidemDaUnder
{\begin{tric}
\draw (162:2)..controls(126:1.6)and(54:1.6)..(18:2);
\draw    [double=darkblue,ultra thick,white,line width=3pt]                 (90:0)--(90:2) 
           (306:2)--(90:0)--(234:2) ;
\end{tric}
}

\def\ReidemDbUnder
{\begin{tric}
\draw (162:2)..controls(234:1.7)and(306:1.7)..(18:2);
\draw     [double=darkblue,ultra thick,white,line width=3pt]
           (90:0)--(90:2) 
           (306:2)--(90:0)--(234:2) ;
\end{tric}
}

\def\ReideAhalfother
{\begin{tric}
\draw[scale=0.7] (-1.5,-1.5)..controls(-1.5,-1)and(-1,-0.5)..(0,-0.5)
                 (1.5,-1.5)..controls(1.5,-1)and(1,-0.5).. (0,-0.5)
                 (0,-0.5)--(0,1)node[left,black,midway,scale=0.7]{$2$}; 

\draw[scale=0.7]  (1.5,-1.5)..controls(1.5,-3)and(-1.5,-3) ..(-1.5,-4.5) 
node[below,black,scale=0.7]{$1$};
\draw[scale=0.7, double=darkblue,ultra thick,white,line width=3pt]  (-1.5,-1.5)..controls(-1.5,-3)and(1.5,-3) ..(1.5,-4.5)node[below,black,scale=0.7]{$1$};
\end{tric}
}

\def\ReideB
{\begin{tric}
\draw [scale=0.7] (1.5,1.5)..controls(1.5,0)and(-1.5,0) ..(-1.5,-1.5) ;
\draw[scale=0.7, double=darkblue,ultra thick,white,line width=3pt] (-1.5,1.5)..controls(-1.5,0)and(1.5,0) ..(1.5,-1.5) ;

\draw[scale=0.7] (-1.5,-1.5)..controls(-1.5,-3)and(1.5,-3) ..(1.5,-4.5) ;
\draw[scale=0.7, double=darkblue,ultra thick,white,line width=3pt] (1.5,-1.5)..controls(1.5,-3)and(-1.5,-3) ..(-1.5,-4.5);
\end{tric}
}

\def\ReideBmirror
{\begin{tric}
\draw [scale=0.7] (-1.5,1.5)..controls(-1.5,0)and(1.5,0) ..(1.5,-1.5) ;
\draw[scale=0.7, double=darkblue,ultra thick,white,line width=3pt] (1.5,1.5)..controls(1.5,0)and(-1.5,0) ..(-1.5,-1.5) ;

\draw[scale=0.7] (1.5,-1.5)..controls(1.5,-3)and(-1.5,-3) ..(-1.5,-4.5) ;
\draw[scale=0.7, double=darkblue,ultra thick,white,line width=3pt] (-1.5,-1.5)..controls(-1.5,-3)and(1.5,-3) ..(1.5,-4.5) ;
\end{tric}
}

\def\ReideBa
{\begin{tric}
\draw  [scale=0.7] 
(-2,2)..controls(-1,1)and(-1,-1)..(-2,-2) 
(2,-2)..controls(1,-1)and(1,1)..(2,2) ;
\end{tric}
}

\def\SmallTrianglea
{\begin{tric}
\draw [scale=0.7] 
(1.16,-2) node[below,black,scale=0.7]{$1$}
(1.16,-2)--(0,0) 
(-1.16,-2) node[below,black,scale=0.7]{$1$}
(0,0)--(-1.16,-2)
(0,0)--(0,2) node[left,midway,black,scale=0.7]{$2$};
\end{tric}
}

\def\String
{\begin{tric}
\draw  (0,1.5)--(0,-1.5);
\end{tric}
}

 \begin{align}
    \ReideB = \ \ReideBa  \ \ \ \ \ \ &, \ \ \ \ \ \ \ \ReideBmirror = \ \ReideBa  \label{ReideTwo} \\
   \ReidemCa &= \ReidemCb \label{ReideThree} \\
\ReidemDa &= \ \ \ReidemDb \ \  .  \label{ReidemPar}
\end{align}


\section{Type C power sum polynomials}

\subsection{Elementary symmetric and power sum polynomials in Type $A$} 

We recall here the basic definitions of elementary symmetric polynomials, as well as the power elementary polynomials described in \cite{BH24}.

Suppose that $X=\{x_1,...,x_{2n}\}$ is a set of $2n$ commuting variables. Let $X^{(N)}=\{x_1^N, \dots, x_{2n}^N\}$ denote the set of $N\text{-th}$ powers of the variables of $X.$ For each $k=1,...,2n$ we define the elementary symmetric polynomial $E_{k,n}$ in the variables of $X$ to be the sum of all distinct $k\text{-fold}$ products of variables of $X:$

\begin{equation}
E_{k,n}(X)=\sum_{1\leq i_1<i_2< \cdots < i_k \leq 2n} x_{i_1}x_{i_2} \cdots x_{i_k}.
\end{equation}

Similarly, we define

\begin{equation}
E_{k,n}^{(N)}(X):=E_{k,n}(X^{(N)})=\sum _{1\leq i_1<i_2< \cdots < i_k \leq 2n} x_{i_1}^Nx_{i_2}^N \cdots x_{i_k}^N.
\end{equation}

Further, we use the convention that $E_{0,n}(X)=1$ and $E_{i,n}(X)=0$ whenever $i< 0.$ We have been using $n$ instead of $2n$ in the subscript of $E_{k,n}$ to align our notation with our upcoming discussion of the $C_n$ case.

For each $k=1,...,2n,$ the power elementary polynomial $P_{2n}^{(N,k)}\in \mathbb{Z}[e_1, e_2, \dots, e_{2n}]$ from \cite{BH24} is the unique polynomial satisfying

\begin{equation}\label{Type A power sum}
P_{2n}^{(N,k)}(E_{1,n}(X),E_{2,n}(X), \dots,E_{2n,n}(X))=E_{k,n}^{(N)}(X).
\end{equation}

If we view $x_1,...,x_{2n}$ as eigenvalues of a matrix $A$ in $GL_{2n}(\mathbb{C})$, then $E_{1,n}(X)=\text{tr}(A)$ and $E_{2n,n}(X)=\det(A)$ while $E_{1,n}(X^{(N)})=\text{tr}(A^N).$

Consider the ring given by the image of the quotient
\begin{align*}
\widehat{}: \mathbb{Z}[x_1,\dots x_{2n}] &\rightarrow \mathbb{Z}[x_1,\dots,x_{2n}]/(x_1x_2 \cdots x_{2n}=1)\\
x_i & \mapsto \widehat{x}_i.
\end{align*}
 Let $\widehat{X}=\{\widehat{x}_1,\dots,\widehat{x}_{2n}\}.$ Since $\widehat{x}_{1} \widehat{x}{_2} \cdots \widehat{x}_{2n}=1,$ we can view the variables $\widehat{x}_{i}$ as eigenvalues of a matrix $A$ in $SL_{2n}(\mathbb{C}).$ In this case, $E_{2n,n}(\widehat{X})=1$ and we define the power elementary polynomial $\hat{P}_{2n}^{(N,k)} \in \mathbb{Z}[e_1,e_2,\dots,e_{2n-1}]$ to be

\begin{equation}
\hat{P}_{2n}^{(N,k)}(e_{1},\cdots,e_{2n-1})=P_{2n}^{(N,k)}(e_{1},\cdots,e_{2n-1},1),
\end{equation}

which satisfies

\begin{equation}
\hat{P}_{2n}^{(N,k)}(E_{1,n}(\widehat{X}), E_{2,n}(\widehat{X}), \dots E_{2n-1,n}(\widehat{X}))=E_{k,n}^{(N)}(\widehat{X})
\end{equation}

for each $k=1,\dots,2n-1.$

The elementary symmetric polynomials $E_{i,n}(\widehat{X})$ may be viewed as the character polynomials for $SL_{2n}(\mathbb{C})$ as they record the weight space decompositions of the fundamental representations given by the $i\text{-th}$ exterior powers of the vector representation $\mathbb{C}^{2n}.$

\subsection{Character and power fundamental polynomials in type $C$}

We will describe how to replace the polynomials $E_{i,n}$ and $\hat{P}_{2n}^{(N,k)}$ from the type $A$ case with polynomials $F_{i,n}$ and $Q_{2n}^{(N,k)}$ suitable for type $C.$ The algebraic group $Sp_{2n}(\mathbb{C})$ is a rank $n$ algebraic subgroup of $SL_{2n}(\mathbb{C})$. For a matrix $A$ in $Sp_{2n}(\mathbb{C}) \leq SL_{2n}(\mathbb{C}),$ we can view its eigenvalues as being modeled by variables in the set $\Lambda=\{\lambda_1^{\pm 1}, \dots, \lambda_n^{\pm 1}\}.$ We will consider the ring $\mathbb{Z}[\Lambda]:=\mathbb{Z}[\lambda_1^{\pm 1}, \dots, \lambda_n^{\pm 1}].$

The role that the exterior powers of $\mathbb{C}^{2n}$ played for $SL_{2n}(\mathbb{C})$ are now played by the fundamental representations of $Sp_{2n}(\mathbb{C}).$ If $V_{\varpi_i}$ is the $i\text{-th}$ fundamental representation of $Sp_{2n}(\mathbb{C}),$ then we have the following isomorphism of $Sp_{2n}$ representations: $\wedge^i(\mathbb{C}^{2n})=V_{\varpi_i}\oplus \wedge^{i-2}(\mathbb{C}^{2n}).$ Consequently, if $F_{i,n}\in \mathbb{Z}[\Lambda]$ represents the character polynomial of $V_{\varpi_i},$ we are motivated to obtain it from the elementary symmetric polynomials according to the following definition.

\begin{definition}\label{F_k,n}
For $1\leq i\leq n$ the polynomial $F_{k,n}\in \mathbb{Z}[\Lambda]$ is given by

\begin{equation}\label{F_k,n equation}
F_{k,n}(\lambda_1^{\pm 1},\dots,\lambda_n^{\pm 1})=E_{k,n}(\lambda_1^{\pm 1},\dots,\lambda_n^{\pm 1})-E_{k-2,n}(\lambda_1^{\pm 1},\dots,\lambda_n^{\pm 1}).     
\end{equation}
\end{definition}

For the rest of the paper we will use the shorthand $E_{i,n}:=E_{i,n}(\lambda_1^{\pm 1},...,\lambda_n^{\pm 1})$ and $F_{i,n}:=F_{i,n}(\lambda_1^{\pm 1},...,\lambda_n^{\pm 1}).$ Note the convention that $E_{0,n}=1$ and $E_{i,n}=0$ for all $i <0$ gives the convention that $F_{0,n}=1$ and $F_{i,n}=0$ for all $i<0.$

Note that Definition \ref{F_k,n} implies that $F_{1,n}=E_{1,n}.$ One can then use the relation $E_{k,n}=F_{k,n}+E_{k-2,n}$ to solve inductively for the polynomials $E_{i,n}$ in terms of only the polynomials $F_{i,n}$ to obtain

\begin{equation}\label{E in terms of F}
E_{i,n}=F_{i,n}+F_{i-2,n}+F_{i-4,n}+\cdots+F_{i-2 \lfloor\frac{i}{2} \rfloor,n }.
\end{equation}

Since $\prod_{\lambda \in \Lambda} \lambda=1,$ we can use the existence of the polynomials $\hat{P}_{2n}^{(N,k)}$ and the relationship between the polynomials $F_{i,n}$ and $E_{j,n}$ to construct Type $C$ analogues of power sum polynomials $Q_{2n}^{(N,k)} \in \mathbb{Z}[f_1,...,f_n].$

\begin{definition}\label{Type C polynomial}
For $k=1,\dots, n,$ the Type $C$ power fundamental polynomial $Q_{2n}^{(N,k)}(f_1,\dots,f_n)\in \mathbb{Z}[f_1,\dots,f_n]$ is defined by
\begin{equation*}
Q_{2n}^{(N,k)}(f_1,\dots,f_n)=\hat{P}_{2n}^{(N,k)}(\tilde{f}_1, \dots,\tilde{f}_{n-1},\tilde{f}_n,\tilde{f}_{n-1},\dots\tilde{f}_1)-\hat{P}_{2n}^{(N,k-2)}(\tilde{f}_1, \dots,\tilde{f}_{n-1},\tilde{f}_n,\tilde{f}_{n-1},\dots\tilde{f}_1),
\end{equation*}
where $\tilde{f}_i:=f_i+f_{i-2}+f_{i-4}+\cdots+ f_{i-2\lfloor\frac{i}{2}\rfloor}.$

\end{definition}

We next verify that the polynomial $Q_{2n}^{(N,k)}$ satisfies a property analogous to the main property of the polynomial $\hat{P}_{2n}^{(N,k)}.$

\begin{proposition}\label{Property of Type C polynomial}
The Type $C$ power fundamental polynomial $Q_{2n}^{(N,k)}(f_1,...,f_n)$ is the unique polynomial in $\mathbb{Z}[f_1,...,f_n]$ satisfying the property that

\begin{equation}\label{power sum}
Q_{2n}^{(N,k)}(F_{1,n},...,F_{n,n})=F_{k,n}^{(N)}(\Lambda),
\end{equation}    
\end{proposition}

\begin{proof}
The uniqueness follows from the algebraic independence of the $F_{i,n},$ which in turns follows from the algebraic independence of the $E_{i,n}.$ The fact that $Q_{2n}^{(N,k)}$ has integer coefficients follows from the facts that $\hat{P}_{2n}^{(N,k)}\in \mathbb{Z}[e_1,\dots, e_{2n-1}]$ has integer coefficients and that for each $i,$ $\tilde{f_i} \in \mathbb{Z}[f_1,\dots,f_n]$. 

We then check that our construction of $Q_{2n}^{(N,k)}$ in Definition \ref{Type C polynomial} satisfies the desired property (\ref{power sum}). By definition, we have that

\begin{equation*}
Q_{2n}^{(N,k)}(f_1,\dots, f_n)= \hat{P}_{2n}^{(N,k)}(\tilde{f}_1, \dots, \tilde{f}_n, \tilde{f}_{n-1}, \dots, \tilde{f}_1)- \hat{P}_{2n}^{(N,k-2)}(\tilde{f}_1, \dots, \tilde{f}_n, \tilde{f}_{n-1}, \dots, \hat{f}_1).\\
\end{equation*}

Evaluating at $(f_1,\dots,f_n)=(F_{1,n}, \dots, F_{n,n})$ yields

\begin{align*}
Q_{2n}^{(N,k)}(F_{1,n},\dots, F_{n,n})&\overset{(\ref{E in terms of F})}{=}\hat{P}_{2n}^{(N,k)}(E_{1,n},\dots ,E_{n,n}, E_{n-1,n}, \dots, E_{1,n})-\hat{P}_{2n}^{(N,k-2)}(E_{1,n}, \dots, E_{n,n}, E_{n-1,n}, \dots E_{1,n})\\
&=\hat{P}_{2n}^{(N,k)}(E_{1,n}, \dots, E_{2n-1,n})-\hat{P}_{2n}^{(N,k-2)}(E_{1,n}, \dots, E_{2n-1,n})\\
&\overset{(\ref{Type A power sum})}{=}E_{k,n}^{(N)}-E_{k-2,n}^{(N)}\\
&\overset{(\ref{F_k,n equation})}{=}F_{k,n}^{(N)},
\end{align*}

as required. In the second equality we used the fact that $E_{i,n}=E_{2n-i,n}$ for all $0 \leq i \leq n,$ since $\prod_{\lambda \in \Lambda} \lambda=1.$
\end{proof}

\subsection{Relations between $C_n$ and $C_{n-1}$ character polynomials}

Our overall goal is to use the polynomials $Q_{2n}^{(N,k)}$ to construct central elements in $C_n$ skein algebras. Our next step toward this goal is observing some identities among the polynomials $F_{i,n}$ which will later manifest themselves in corresponding skein identities. We first record the following observation, which will be useful to us in the subsequent Lemma \ref{evalue relations}. 

\begin{lemma}\label{fn,n-1}
For a set of $2n-2$ commuting formal variables of the form $\Lambda_{n-1}=\{\lambda_1,...,\lambda_{n-1},\lambda_1^{-1},...,\lambda_{n-1}^{-1}\}$, we have \begin{equation*}E_{n,n-1}(\Lambda_{n-1})=E_{n-2,n-1}(\Lambda_{n-1}).\end{equation*}
Consequently, \begin{equation*} F_{n,n-1}(\lambda_1^{\pm 1},...,\lambda_{n-1}^{\pm 1})=0. \end{equation*}
\end{lemma}

\begin{proof}

The fact that the first statement implies that $F_{n,n-1}(\lambda_1^{\pm 1},...,\lambda_{n-1}^{\pm 1})=0$ is an immediate consequence of Equation (\ref{F_k,n equation}) defining $F_{n,n-1}.$ We now focus on showing the first statement.

Rename the variables of $\Lambda_{n-1}$ in the following manner:

\begin{equation*}x_i=\begin{cases}
\lambda_i & i\leq n-1\\
\lambda_{i-n+1}^{-1} & n \leq i \leq 2n-2.
\end{cases}
\end{equation*}

Let $I=\{1,2,\dots,2n-2\}$ and let $J=\{j_1,j_2, \dots, j_k\} \subset I$ be any subset of some cardinality $k$. Denote by $x_J$ the monomial given by $x_J=x_{j_1}x_{j_2}\cdots x_{j_k}.$ 

By definition of the elementary symmetric polynomial $E_{n-2,n-1},$ we have that

\begin{equation*}
E_{n-2,n-1}(\Lambda_{n-1})=\sum_{J \subset I, \vert J \vert=n-2} x_J.
\end{equation*}

Since for each $J \subset I$ with $\vert J \vert =n-2$ we have $\vert I \setminus J \vert =n,$ we see that $E_{n,n-1}$ can be written

\begin{equation*}
E_{n,n-1}(\Lambda_{n-1})=\sum_{J \subset I, \vert J \vert =n-2} x_{I \setminus J}.
\end{equation*}

Since $x_1 x_2 \cdots x_{2n-2}=1,$ we have that $x_{I\setminus J}=(x_J)^{-1}$ for each $J \subset I.$ Since the operation of taking inverses of variables is a bijection on $\Lambda_{n-1}$, we deduce that

\begin{equation*}
\sum_{J \subset I, \vert J \vert =n-2} (x_J)^{-1}=\sum_{J \subset I, \vert J \vert =n-2} x_J.
\end{equation*}

Thus,

\begin{align*}
E_{n,n-1}(\Lambda_{n-1}) &=\sum_{J \subset I, \vert J \vert =n-2} x_{I \setminus J}\\
&=\sum_{J \subset I, \vert J \vert =n-2} (x_J)^{-1}\\
&=\sum_{J \subset I, \vert J \vert =n-2} x_J\\
&=E_{n-2,n-1}(\Lambda_{n-1}),
\end{align*}
as claimed.
\end{proof}

For the following, for $l \geq 1$ let $F_{k,l}:=F_{k,l}(\lambda_1^{\pm 1}, \dots, \lambda_l^{\pm 1})$. We prove the following key lemma, which establishes recursive relations among the $F_{i,j}$ in equations which will later be shown to have diagrammatic analogues in our key Theorem \ref{MainRelations}.

\begin{lemma}\label{evalue relations}
For all $n \geq 1$ we have

\begin{align*}
F_{1,n}&=\lambda_n+\lambda_n^{-1}+F_{1,n-1},\\
F_{k,n}&=(\lambda_n+\lambda_n^{-1})F_{k-1,n-1}+F_{k-2,n-1}+F_{k,n-1},\text{ for } 2 \leq k \leq n-1,\\
F_{n,n}&=(\lambda_n+\lambda_n^{-1})F_{n-1,n-1}+F_{n-2,n-1}.
\end{align*}
\end{lemma}

\begin{proof}
By definition, for each $k,l$ we have $F_{k,l}=E_{k,l}-E_{k-2,l}.$ By separating the terms of $E_{k,n}$ into those containing exactly one of $\lambda_n$ or $\lambda_n^{-1},$ those containing both of $\lambda_n$ and $\lambda_n^{-1}$, or those containing neither $\lambda_n$ nor $\lambda_n^{-1},$ we obtain the following identity

\begin{equation*}
E_{k,n}=(\lambda_n+\lambda_n^{-1})E_{k-1,n-1}+E_{k-2,n-1}+E_{k,n-1}.
\end{equation*}

Similarly
\begin{equation*}
E_{k-2,n}=(\lambda_n+\lambda_n^{-1})E_{k-3,n-1}+E_{k-4,n-1}+E_{k-2,n-1}.
\end{equation*}

We then have that

\begin{align*}
F_{k,n}&=E_{k,n}-E_{k-2,n}\\
&=(\lambda_n+\lambda_n^{-1})(E_{k-1,n-1}-E_{k-3,n-1})+E_{k-2,n-1}-E_{k-4,n-1}+E_{k,n-1}-E_{k-2,n-1}\\
&=(\lambda_n+\lambda_n^{-1})F_{k-1,n-1}+F_{k-2,n-1}+F_{k,n-1}.
\end{align*}

This proves the claim for $2 \leq k \leq n-1.$ Now, for the special case of $k=1,$ we observe that $F_{0,n-1}=1$ and $F_{-1,n-1}=0$ by convention and we obtain the identity as claimed. Lastly, for the special case of $k=n$ we have that $F_{n,n-1}=0$ by Lemma \ref{fn,n-1}, and we obtain the claimed identity.
\end{proof}

\section{Annuli and transparent elements}\label{transparent section}

\subsection{Graphical calculations in the annulus}

Consider the annulus $\mathbb{A}=S^1\times [0,1]$ and a fixed point $p \in S^1$. Let $p_0=(p,0)$ and $p_1=(p,1)$ be two marked points on $\partial \mathbb{A}.$ We call $(\mathbb{A},\{p_0,p_1\})$ the marked annulus. We soon consider webs with boundary equal to $\{p_0,p_1\}.$

\begin{definition}
Define $\mathcal{A}$ to be the $\mathcal{R}\text{-algebra}$ spanned by closed $C_n$ webs in $\mathbb{A}$. Let $s_1:S^1 \times [0,1] \rightarrow S^1 \times [1/2,1]$ be given by $s_1(x,t)=(x,(t+1)/2)$ and let $s_2: S^1 \times [0,1] \rightarrow S^1 \times [0,1/2]$ be given by $s_2(x,t)=(x,t/2).$ If $W_1,W_2 \in \mathcal{A}$ are two webs, we define their product, $W_1W_2$ to be the web obtained by taking $s_1(W_1) \cup s_2(W_2)$. This product agrees with that of $\mathcal{S}_q^{C_n}(\mathbb{A})$.
\end{definition}

\begin{definition}
Define $\mathcal{A}_{1,1}$ to be the skein module $\mathcal{S}_q^{C_n}(\mathbb{A},\{p_0,p_1\})$ of the marked annulus spanned by webs with a label $1$ endpoint at $p_0$ on the interior boundary of the annulus and one label $1$ endpoint at $p_1$ on the exterior boundary of the annulus. The algebra structure for $\mathcal{A}_{1,1}$ is given by concatenating annuli outward, as described for $\mathcal{A}.$ Note that this product is not the same as a skein algebra product in this case.
\end{definition}

\begin{definition}\label{star homomorphisms}
We define an algebra homomorphism denoted by $(-)^{\star}: \mathcal{A} \rightarrow \mathcal{A}_{1,1}$ such that if $D \in \mathcal{A}$ is a diagram, the element $D^{\star} \in \mathcal{A}_{1,1}$ is the diagram of $D$ placed over the identity strand $\text{id} \in \mathcal{A}_{1,1}.$ It can be checked diagrammatically that this is a well-defined algebra homomorphism.

Similarly, we define $(-)_\star: \mathcal{A} \rightarrow \mathcal{A}_{1,1}$ to be the algebra homomorphism such that if $D \in \mathcal{A}$ is a diagram, the element $D_{\star} \in \mathcal{A}_{1,1}$ is the diagram of $D$ placed below the identity strand $\text{id} \in \mathcal{A}_{1,1}.$
\end{definition}

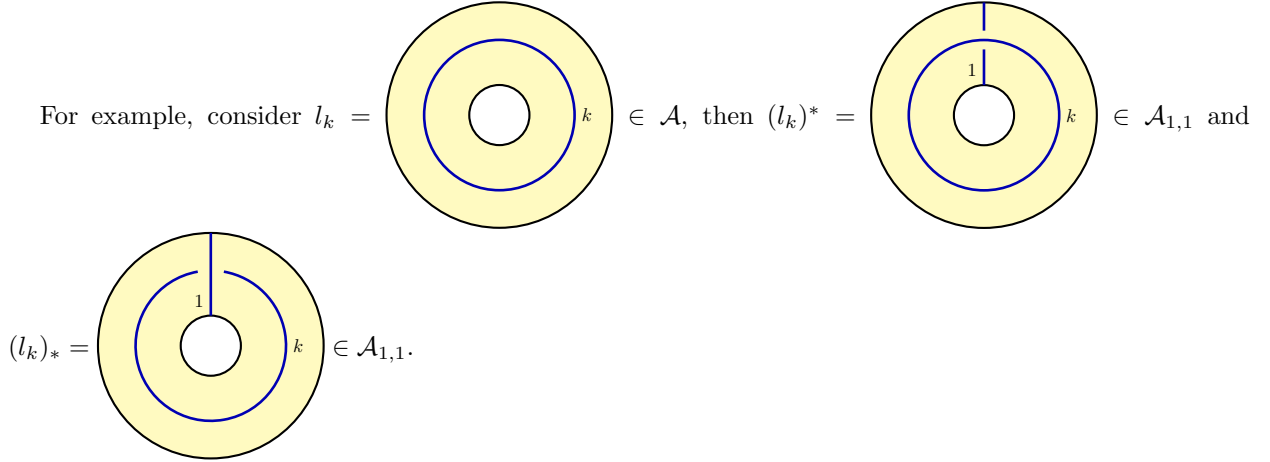
\def\KLoop{
\begin{tric}
   \draw [thick,black,fill=yellow!30] (0,0) circle (3cm);
   \draw [thick,black,fill=white] (0,0) circle (0.8cm);
   \draw (0,0) circle (2cm);

   \draw (2,0)node[right,black,scale=0.7]{$k$}; 
\end{tric}
}

\def\KOverLoop{
\begin{tric}
   \draw [thick,black,fill=yellow!30] (0,0) circle (3cm);
   \draw [thick,black,fill=white] (0,0) circle (0.8cm);
   \draw (0,0) circle (2cm);
   \draw (0,0.8)--(0,1.75) (0,2.25)--(0,3); 

   \draw (0,1.2)node[left,black,scale=0.7]{$1$}
         (2,0)node[right,black,scale=0.7]{$k$}; 
\end{tric}
}

\def\KUnderLoop{
\begin{tric}
   \draw [thick,black,fill=yellow!30] (0,0) circle (3cm);
   \draw [thick,black,fill=white] (0,0) circle (0.8cm);
   \draw (100:2cm) arc (100:440:2cm);
   \draw (0,0.8)--(0,3); 

   \draw (0,1.2)node[left,black,scale=0.7]{$1$}
         (2,0)node[right,black,scale=0.7]{$k$}; 
\end{tric}
}

For example, consider $l_k=\KLoop  \in \mathcal{A} $, then $(l_k)^*=\KOverLoop \in \mathcal{A}_{1,1}$ and $(l_k)_*=\KUnderLoop \in \mathcal{A}_{1,1}$.

\def\VertexDef{
\begin{tric}
   \draw [thick,black,fill=yellow!30] (0,0) circle (3cm);
   \draw [thick,black,fill=white] (0,0) circle (0.7cm);
   \draw (0,0.7)--(0,3);
   \draw (0,0) circle (1.85cm);
   \draw[fill=green,draw=darkgreen] (0,1.85)  circle (0.35cm); 

   \draw (-0.1,2.6)node[right,black,scale=0.7]{$1$}
 (-0.1,1.1)node[right,black,scale=0.7]{$1$}
  (0,1.85)node[black,scale=0.7]{$k$};
\end{tric}
}

\def\VertexDefL{
\begin{tric}
   \draw [thick,black,fill=yellow!30] (0,0) circle (3cm);
   \draw [thick,black,fill=white] (0,0) circle (0.55cm);
   \draw (0,0.55)--(0,3);
   \draw (0,0) circle (1.775cm);
   \draw[fill=green,draw=darkgreen] (0,1.775)  circle (0.63cm); 

   \draw (-0.1,2.7)node[right,black,scale=0.7]{$1$}
 (-0.1,0.8)node[right,black,scale=0.7]{$1$}
  (0,1.775)node[black,scale=0.7]{$k-1$};
\end{tric}
}

\def\OneVertexDef{
\begin{tric}
   \draw [thick,black,fill=yellow!30] (0,0) circle (3cm);
   \draw [thick,black,fill=white] (0,0) circle (0.7cm);
   \draw (0,0.7)--(0,3);
   \draw (0,0) circle (1.85cm);
   \draw[fill=green,draw=darkgreen] (0,1.85)  circle (0.35cm); 

   \draw (-0.1,2.6)node[right,black,scale=0.7]{$1$}
 (-0.1,1.1)node[right,black,scale=0.7]{$1$}
  (0,1.85)node[black,scale=0.7]{$1$};
\end{tric}
}

\def\LowVertexDef{
\begin{tric}
   \draw [thick,black,fill=yellow!30] (0,0) circle (3cm);
   \draw [thick,black,fill=white] (0,0) circle (0.55cm);
   \draw (0,0.55)--(0,3);
   \draw (0,0) circle (1.775cm);
   \draw[fill=green,draw=darkgreen] (0,1.775)  circle (0.63cm); 

   \draw (-0.1,2.7)node[right,black,scale=0.7]{$1$}
 (-0.1,0.8)node[right,black,scale=0.7]{$1$}
  (0,1.775)node[black,scale=0.7]{$k-2$};
\end{tric}
}

\def\NLowVertexDef{
\begin{tric}
   \draw [thick,black,fill=yellow!30] (0,0) circle (3cm);
   \draw [thick,black,fill=white] (0,0) circle (0.55cm);
   \draw (0,0.55)--(0,3);
   \draw (0,0) circle (1.775cm);
   \draw[fill=green,draw=darkgreen] (0,1.775)  circle (0.63cm); 

   \draw (-0.1,2.7)node[right,black,scale=0.7]{$1$}
 (-0.1,0.8)node[right,black,scale=0.7]{$1$}
  (0,1.775)node[black,scale=0.7]{$n-2$};
\end{tric}
}

\def\VertexDefa{
\begin{tric}
   \draw [thick,black,fill=yellow!30] (0,0) circle (3cm);
   \draw (-0.5,2)--(0.7,1.5) 
         (0.7,1.5)--(0,0.8) 
         (0.28,1)node[right,black,scale=0.7]{$1$}
         (-0.5,2)--(0,3) node[left,midway,black,scale=0.7]{$1$};
   \draw (-0.2,2.05)node[right,black,scale=0.7]{$k+1$};  
   
   \draw (-0.5,2)..controls(-2.5,2)and(-3,-2)..(0,-2)
         (0.7,1.5)..controls(2.5,1.5)and(2.5,-2)..(0,-2);
   \draw(1.95,0)node[right,black,scale=0.7]{$k$}; 
    
   \draw [thick,black,fill=white] (0,0) circle (0.8cm);
\end{tric}
}

\def\VertexDefaIso{
\begin{tric}
   \draw [thick,black,fill=yellow!30] (0,0) circle (3.5cm);
   \draw (-0.5,2)--(0.7,1.5) 
         (0.7,1.5)--(0,0.8) 
         (0.28,1)node[right,black,scale=0.7]{$1$};
   \draw (-0.2,2.05)node[right,black,scale=0.7]{$k$};  
   
   \draw (-0.5,2)..controls(-2.5,2)and(-3,-2)..(0,-2)
         (0.7,1.5)..controls(2.5,1.5)and(2.5,-2)..(0,-2);
   \draw(0,-1.95)node[above,black,scale=0.7]{$k-1$}; 
    
   \draw [thick,black,fill=white] (0,0) circle (0.8cm);

   \draw (-0.5,2)..controls(0,3)and(-3,3)..(-3,0);
   \draw (-3,0)..controls(-3,-4)and(3,-4)..(3,0);
   \draw (3,0)..controls(3,2)and(0,2)..(0,3.5) node[above,midway,black,scale=0.7]{$1$}; 
\end{tric}
}

\def\VertexDefaBr{
\begin{tric}
   \draw [thick,black,fill=yellow!30] (0,0) circle (3cm);
   \draw (-0.5,2)--(0.7,1.5) 
         (0.7,1.5)--(0,0.8) 
         (0.28,1)node[right,black,scale=0.7]{$1$}
         (-0.5,2)--(0,3) node[left,midway,black,scale=0.7]{$1$};
   \draw (-0.2,2.05)node[right,black,scale=0.7]{$k-1$};  
   
   \draw (-0.5,2)..controls(-2.5,2)and(-3,-2)..(0,-2)
         (0.7,1.5)..controls(2.5,1.5)and(2.5,-2)..(0,-2);
   \draw(1.95,0)node[right,black,scale=0.7]{$k$}; 
    
   \draw [thick,black,fill=white] (0,0) circle (0.8cm);
\end{tric}
}

\def\NVertexDefaBr{
\begin{tric}
   \draw [thick,black,fill=yellow!30] (0,0) circle (3cm);
   \draw (-0.5,2)--(0.7,1.5) 
         (0.7,1.5)--(0,0.8) 
         (0.28,1)node[right,black,scale=0.7]{$1$}
         (-0.5,2)--(0,3) node[left,midway,black,scale=0.7]{$1$};
   \draw (-0.2,2.05)node[right,black,scale=0.7]{$n-1$};  
   
   \draw (-0.5,2)..controls(-2.5,2)and(-3,-2)..(0,-2)
         (0.7,1.5)..controls(2.5,1.5)and(2.5,-2)..(0,-2);
   \draw(1.95,0)node[right,black,scale=0.7]{$n$}; 
    
   \draw [thick,black,fill=white] (0,0) circle (0.8cm);
\end{tric}
}

\def\VertexDefaH{
\begin{tric}
   \draw [thick,black,fill=yellow!30] (0,0) circle (3cm);
   \draw (-0.5,2)--(0.7,1.5) 
         (0.7,1.5)--(0,0.8) 
         (0.28,1)node[right,black,scale=0.7]{$1$}
         (-0.5,2)--(0,3) node[left,midway,black,scale=0.7]{$1$};
   \draw (-0.2,2.05)node[right,black,scale=0.7]{$k+2$};  
   
   \draw (-0.5,2)..controls(-2.5,2)and(-3,-2)..(0,-2)
         (0.7,1.5)..controls(2.5,1.5)and(2.5,-2)..(0,-2);
   \draw(0,-1.95)node[below,black,scale=0.7]{$k+1$}; 
    
   \draw [thick,black,fill=white] (0,0) circle (0.8cm);
\end{tric}
}

\def\VertexDefaHL{
\begin{tric}
   \draw [thick,black,fill=yellow!30] (0,0) circle (3cm);
   \draw (-0.5,2)--(0.7,1.5) 
         (0.7,1.5)--(0,0.8) 
         (0.28,1)node[right,black,scale=0.7]{$1$}
         (-0.5,2)--(0,3) node[left,midway,black,scale=0.7]{$1$};
   \draw (-0.2,2.05)node[right,black,scale=0.7]{$k$};  
   
   \draw (-0.5,2)..controls(-2.5,2)and(-3,-2)..(0,-2)
         (0.7,1.5)..controls(2.5,1.5)and(2.5,-2)..(0,-2);
   \draw(0,-1.95)node[below,black,scale=0.7]{$k+1$}; 
    
   \draw [thick,black,fill=white] (0,0) circle (0.8cm);
\end{tric}
}

\def\OneVertexDefa{
\begin{tric}
   \draw [thick,black,fill=yellow!30] (0,0) circle (3cm);
   \draw (-0.5,2)--(0.7,1.5) 
         (0.7,1.5)--(0,0.8) 
         (0.28,1)node[right,black,scale=0.7]{$1$}
         (-0.5,2)--(0,3) node[left,midway,black,scale=0.7]{$1$};
   \draw (-0.2,2.05)node[right,black,scale=0.7]{$2$};  
   
   \draw (-0.5,2)..controls(-2.5,2)and(-3,-2)..(0,-2)
         (0.7,1.5)..controls(2.5,1.5)and(2.5,-2)..(0,-2);
   \draw(1.95,0)node[right,black,scale=0.7]{$1$}; 
    
   \draw [thick,black,fill=white] (0,0) circle (0.8cm);
\end{tric}
}

\def\VertexDefb{
\begin{tric}
   \draw [thick,black,fill=yellow!30] (0,0) circle (3cm);
   \draw [thick,black,fill=white] (0,0) circle (0.55cm);
   \draw (0,0) circle (1.7cm);
   \draw (0,0.55)--(0,1.7)
          (0.1,0.8) node[left,black,scale=0.7]{$1$};
   
   \draw (0,1.9)..controls(0,3.5)and(2.3,2)..(2.3,0);
   \draw (2.3,0)..controls(2.3,-3)and(-2.3,-3)..(-2.3,0);
   \draw (-2.3,0)..controls(-2.3,2)and(-0.5,2)..(0,3) node[above,midway,black,scale=0.7]{$1$}; 

   \draw[fill=green,draw=darkgreen] (0,1.7)  circle (0.63cm);
 \draw (0,1.7)node[black,scale=0.7]{$k-1$};
\end{tric}
}

\def\VertexDefbH{
\begin{tric}
   \draw [thick,black,fill=yellow!30] (0,0) circle (3cm);
   \draw [thick,black,fill=white] (0,0) circle (0.65cm);
   \draw (0,0) circle (1.7cm);
   \draw (0,0.65)--(0,1.7)
        (-0.1,1)node[right,black,scale=0.7]{$1$};
   
   \draw (0,1.7)..controls(0,3)and(2.3,2)..(2.3,0);
   \draw (2.3,0)..controls(2.3,-3)and(-2.3,-3)..(-2.3,0);
   \draw (-2.3,0)..controls(-2.3,2)and(0,2)..(0,3) node[above,midway,black,scale=0.7]{$1$}; 

    \draw[fill=green,draw=darkgreen] (0,1.7)  circle (0.35cm);
    \draw(0,1.7) node[black,scale=0.7]{$k$};
\end{tric}
}

\def\OneVertexDefb{
\begin{tric}
   \draw [thick,black,fill=yellow!30] (0,0) circle (3cm);
   \draw [thick,black,fill=white] (0,0) circle (0.8cm);
   \draw (0,0.8)--(0,1.7)
          (0.1,1.5) node[left,black,scale=0.7]{$1$};
   
   \draw (0,1.7)..controls(0,3)and(2.3,2)..(2.3,0);
   \draw (2.3,0)..controls(2.3,-3)and(-2.3,-3)..(-2.3,0);
   \draw (-2.3,0)..controls(-2.3,2)and(0,2)..(0,3) ;
\end{tric}
}

\def\NVertexDefb{
\begin{tric}
   \draw [thick,black,fill=yellow!30] (0,0) circle (3cm);
   \draw [thick,black,fill=white] (0,0) circle (0.55cm);
   \draw (0,0) circle (1.7cm);
   \draw (0,0.55)--(0,1.7)
          (0.1,0.8) node[left,black,scale=0.7]{$1$};
   
   \draw (0,1.9)..controls(0,3.5)and(2.3,2)..(2.3,0);
   \draw (2.3,0)..controls(2.3,-3)and(-2.3,-3)..(-2.3,0);
   \draw (-2.3,0)..controls(-2.3,2)and(-0.5,2)..(0,3) node[above,midway,black,scale=0.7]{$1$}; 

   \draw[fill=green,draw=darkgreen] (0,1.7)  circle (0.63cm);
 \draw (0,1.7)node[black,scale=0.7]{$n-1$};
\end{tric}
}

\def\RVertexDefa{
\begin{tric}
   \draw [thick,black,fill=yellow!30] (0,0) circle (3cm);
   \draw (0.5,2)--(-0.7,1.5) 
         (-0.7,1.5)--(0,0.8) 
         (-0.28,1)node[left,black,scale=0.7]{$1$}
         (0.5,2)--(0,3) node[right,midway,black,scale=0.7]{$1$};
   \draw (0.2,2.05)node[left,black,scale=0.7]{$k+1$};  
   
   \draw (0.5,2)..controls(2.5,2)and(3,-2)..(0,-2)
         (-0.7,1.5)..controls(-2.5,1.5)and(-2.5,-2)..(0,-2);
   \draw(2.1,0)node[right,black,scale=0.7]{$k$}; 
    
   \draw [thick,black,fill=white] (0,0) circle (0.8cm);
\end{tric}
}

\def\RVertexDefaIso{
\begin{tric}
   \draw [thick,black,fill=yellow!30] (0,0) circle (3.5cm);
   \draw (0.5,2)--(-0.7,1.5) 
         (-0.7,1.5)--(0,0.8) 
         (-0.28,1)node[left,black,scale=0.7]{$1$};
         
   \draw (0.2,2.05)node[left,black,scale=0.7]{$k$};  
   
   \draw (0.5,2)..controls(2.5,2)and(3,-2)..(0,-2)
         (-0.7,1.5)..controls(-2.5,1.5)and(-2.5,-2)..(0,-2);
   \draw(0,-1.95)node[above,black,scale=0.7]{$k-1$}; 
    
   \draw [thick,black,fill=white] (0,0) circle (0.8cm);
   
   \draw (0.5,2)..controls(0,3)and(3,3)..(3,0);
   \draw (3,0)..controls(3,-4)and(-3,-4)..(-3,0);
   \draw (-3,0)..controls(-3,2)and(0,2)..(0,3.5) node[above,midway,black,scale=0.7]{$1$}; 
\end{tric}
}

\def\RVertexDefaBr{
\begin{tric}
   \draw [thick,black,fill=yellow!30] (0,0) circle (3cm);
   \draw (0.5,2)--(-0.7,1.5) 
         (-0.7,1.5)--(0,0.8) 
         (-0.28,1)node[left,black,scale=0.7]{$1$}
         (0.5,2)--(0,3) node[right,midway,black,scale=0.7]{$1$};
   \draw (0.2,2.05)node[left,black,scale=0.7]{$k-1$};  
   
   \draw (0.5,2)..controls(2.5,2)and(3,-2)..(0,-2)
         (-0.7,1.5)..controls(-2.5,1.5)and(-2.5,-2)..(0,-2);
   \draw(2.1,0)node[right,black,scale=0.7]{$k$}; 
    
   \draw [thick,black,fill=white] (0,0) circle (0.8cm);
\end{tric}
}

\def\NRVertexDefaBr{
\begin{tric}
   \draw [thick,black,fill=yellow!30] (0,0) circle (3cm);
   \draw (0.5,2)--(-0.7,1.5) 
         (-0.7,1.5)--(0,0.8) 
         (-0.28,1)node[left,black,scale=0.7]{$1$}
         (0.5,2)--(0,3) node[right,midway,black,scale=0.7]{$1$};
   \draw (0.2,2.05)node[left,black,scale=0.7]{$n-1$};  
   
   \draw (0.5,2)..controls(2.5,2)and(3,-2)..(0,-2)
         (-0.7,1.5)..controls(-2.5,1.5)and(-2.5,-2)..(0,-2);
   \draw(2.1,0)node[right,black,scale=0.7]{$n$}; 
    
   \draw [thick,black,fill=white] (0,0) circle (0.8cm);
\end{tric}
}

\def\RVertexDefaH{
\begin{tric}
   \draw [thick,black,fill=yellow!30] (0,0) circle (3cm);
   \draw (0.5,2)--(-0.7,1.5) 
         (-0.7,1.5)--(0,0.8) 
         (-0.28,1)node[left,black,scale=0.7]{$1$}
         (0.5,2)--(0,3) node[right,midway,black,scale=0.7]{$1$};
   \draw (0.2,2.05)node[left,black,scale=0.7]{$k+2$};  
   
   \draw (0.5,2)..controls(2.5,2)and(3,-2)..(0,-2)
         (-0.7,1.5)..controls(-2.5,1.5)and(-2.5,-2)..(0,-2);
   \draw(0,-2.1)node[below,black,scale=0.7]{$k+1$}; 
    
   \draw [thick,black,fill=white] (0,0) circle (0.8cm);
\end{tric}
}

\def\RVertexDefaHL{
\begin{tric}
   \draw [thick,black,fill=yellow!30] (0,0) circle (3cm);
   \draw (0.5,2)--(-0.7,1.5) 
         (-0.7,1.5)--(0,0.8) 
         (-0.28,1)node[left,black,scale=0.7]{$1$}
         (0.5,2)--(0,3) node[right,midway,black,scale=0.7]{$1$};
   \draw (0.2,2.05)node[left,black,scale=0.7]{$k$};  
   
   \draw (0.5,2)..controls(2.5,2)and(3,-2)..(0,-2)
         (-0.7,1.5)..controls(-2.5,1.5)and(-2.5,-2)..(0,-2);
   \draw(0,-2.1)node[below,black,scale=0.7]{$k+1$}; 
    
   \draw [thick,black,fill=white] (0,0) circle (0.8cm);
\end{tric}
}

\def\RVertexDefb{
\begin{tric}
   \draw [thick,black,fill=yellow!30] (0,0) circle (3cm);
   \draw [thick,black,fill=white] (0,0) circle (0.55cm);
   \draw (0,0) circle (1.7cm);
   \draw (0,0.55)--(0,1.7)
          (0.1,0.8) node[left,black,scale=0.7]{$1$};
   
   \draw (0,1.9)..controls(0,3.5)and(-2.3,2)..(-2.3,0);
   \draw (-2.3,0)..controls(-2.3,-3)and(2.3,-3)..(2.3,0);
   \draw (2.3,0)..controls(2.3,2)and(0.5,2)..(0,3) node[above,midway,black,scale=0.7]{$1$}; 

\draw[fill=green,draw=darkgreen] (0,1.7)  circle (0.63cm);
 \draw (0,1.7)node[black,scale=0.7]{$k-1$};
\end{tric}
}

\def\RVertexDefbH{
\begin{tric}
   \draw [thick,black,fill=yellow!30] (0,0) circle (3cm);
   \draw [thick,black,fill=white] (0,0) circle (0.65cm);
   \draw (0,0) circle (1.7cm);
   \draw (0,0.65)--(0,1.7)
          (0.1,1) node[left,black,scale=0.7]{$1$};
   
   \draw (0,1.7)..controls(0,3)and(-2.3,2)..(-2.3,0);
   \draw (-2.3,0)..controls(-2.3,-3)and(2.3,-3)..(2.3,0);
   \draw (2.3,0)..controls(2.3,2)and(0,2)..(0,3) node[above,midway,black,scale=0.7]{$1$}; 

     \draw[fill=green,draw=darkgreen] (0,1.7)  circle (0.35cm);
    \draw(0,1.7) node[black,scale=0.7]{$k$};
\end{tric}
}

\def\OneRVertexDefb{
\begin{tric}
   \draw [thick,black,fill=yellow!30] (0,0) circle (3cm);
   \draw [thick,black,fill=white] (0,0) circle (0.8cm);
   \draw (0,0.8)--(0,1.7)
          (0.1,1.5) node[left,black,scale=0.7]{$1$};
   
   \draw (0,1.7)..controls(0,3)and(-2.3,2)..(-2.3,0);
   \draw (-2.3,0)..controls(-2.3,-3)and(2.3,-3)..(2.3,0);
   \draw (2.3,0)..controls(2.3,2)and(0,2)..(0,3) ; 
\end{tric}
}

\def\NRVertexDefb{
\begin{tric}
   \draw [thick,black,fill=yellow!30] (0,0) circle (3cm);
   \draw [thick,black,fill=white] (0,0) circle (0.55cm);
   \draw (0,0) circle (1.7cm);
   \draw (0,0.55)--(0,1.7)
          (0.1,0.8) node[left,black,scale=0.7]{$1$};
   
   \draw (0,1.9)..controls(0,3.5)and(-2.3,2)..(-2.3,0);
   \draw (-2.3,0)..controls(-2.3,-3)and(2.3,-3)..(2.3,0);
   \draw (2.3,0)..controls(2.3,2)and(0.5,2)..(0,3) node[above,midway,black,scale=0.7]{$1$}; 

\draw[fill=green,draw=darkgreen] (0,1.7)  circle (0.63cm);
 \draw (0,1.7)node[black,scale=0.7]{$n-1$};
\end{tric}
}

\def\AnnBraid{
\begin{tric}
   \draw [thick,black,fill=yellow!30] (0,0) circle (3cm);
   \draw [thick,black,fill=white] (0,0) circle (0.8cm);
   \draw (0,0) circle (2cm);
   \draw (0,0.8)--(0,1.75) (0,2.25)--(0,3); 

   \draw (0,1.2)node[left,black,scale=0.7]{$1$}
         (2,0)node[right,black,scale=0.7]{$k$}; 
\end{tric}
}

\def\OneAnnBraid{
\begin{tric}
   \draw [thick,black,fill=yellow!30] (0,0) circle (3cm);
   \draw [thick,black,fill=white] (0,0) circle (0.8cm);
   \draw (0,0) circle (2cm);
   \draw (0,0.8)--(0,1.75) (0,2.25)--(0,3); 

   \draw (0,1.2)node[left,black,scale=0.7]{$1$}
         (2,0)node[right,black,scale=0.7]{$1$}; 
\end{tric}
}

\def\NAnnBraid{
\begin{tric}
   \draw [thick,black,fill=yellow!30] (0,0) circle (3cm);
   \draw [thick,black,fill=white] (0,0) circle (0.8cm);
   \draw (0,0) circle (2cm);
   \draw (0,0.8)--(0,1.75) (0,2.25)--(0,3); 

   \draw (0,1.2)node[left,black,scale=0.7]{$1$}
         (2,0)node[right,black,scale=0.7]{$n$}; 
\end{tric}
}

\begin{notation}\label{notation} We collect the following special elements of $\mathcal{A}_{1,1}$ for future reference.
   \begin{equation*} T= \OneVertexDefb, T^{-1} =\OneRVertexDefb, X_k=\VertexDef, \text{and } l_k^*= \AnnBraid. 
   \end{equation*}

By convention, we let $X_0=\text{id} \in \mathcal{A}_{1,1}.$ The definition of $X_k$ for $1 \leq k \leq n-1$ is given next.
\end{notation}

\begin{definition}\label{VertDef}
We define the elements $X_k \in \mathcal{A}_{1,1}$ inductively by the following equations.
\begin{equation*}
 X_1=\ \OneVertexDef= - \OneVertexDefa - \frac{[n-1]}{[n]} \OneVertexDefb. 
\end{equation*}

For $2\le k \le n-1$,
\begin{equation}
 X_k=\ \VertexDef= (-1)^k \VertexDefa - \frac{[n-k]}{[n-k+1]} \VertexDefb . \label{VertDefInd}
\end{equation}

Here, what is drawn as an unlabeled strand in the diagram for $X_k$ is not actually a strand. Rather, it represents a placeholder for a linear combination of webs.
\end{definition}

We next prove a property of these elements $X_k$ that we will use in the proof of Theorem \ref{MainRelations}.

\begin{lemma}\label{4vertexsym}
The elements $X_k$ from Definition \ref{VertDef} are preserved under the mirror image operation on the annulus given by reversing the orientation of $S^1.$ In other words, for $k\ge 1 $, 
      \begin{align}
      X_k= (-1)^k \VertexDefa - \frac{[n-k]}{[n-k+1]} \VertexDefb \notag\\ 
        = (-1)^k \RVertexDefa - \frac{[n-k]}{[n-k+1]} \RVertexDefb.  \label{Reflect}
      \end{align} 
\end{lemma}

\begin{proof}
  We will proceed by induction. When $k=1$, Equation \eqref{Reflect} follows immediately from Equation \eqref{deducedskein}. Denote the second expression in \eqref{Reflect} by $X_k'.$ For our inductive hypothesis, suppose that $X_k=X_k'.$ We will show that $X_{k+1}=X_{k+1}'.$ We first record a few identities. First, by Equation \eqref{deducedskein} with $k$ replaced by $k+1$ we have,  
  \begin{align}\label{ref1}
    \VertexDefaH
      =\RVertexDefaH +  \frac{[n-k-1]}{[n-k]}\RVertexDefaHL -  \frac{[n-k-1]}{[n-k]}\VertexDefaHL.
  \end{align}  
  
Secondly, by expanding $TX_k$ by using the expression for $X_k'$ in Equation (\ref{Reflect}) and isotopies we obtain
  \begin{align}\label{ref2}
   \VertexDefbH = (-1)^k \VertexDefaHL - \frac{[n-k]}{[n-k+1]} \VertexDefL.
  \end{align}

Thirdly, by expanding $T^{-1}X_k$ by using the definition of $X_k$ from Equation (\ref{VertDefInd}) and applying isotopies we obtain
 \begin{align}\label{ref3}
     \RVertexDefbH = (-1)^k \RVertexDefaHL - \frac{[n-k]}{[n-k+1]} \VertexDefL.
 \end{align} 

Now, we will compute that $X_{k+1}=X_{k+1}'.$ We begin with the first expression of \eqref{Reflect} with $k$ replaced by $k+1$ and apply both \eqref{ref1} and \eqref{ref2} to give
\begin{align*}
& X_{k+1} =  (-1)^{k+1}\VertexDefaH -   \frac{[n-k-1]}{[n-k]} \VertexDefbH\\
  =  & (-1)^{k+1}\left( \RVertexDefaH +  \frac{[n-k-1]}{[n-k]}\RVertexDefaHL \right) + \frac{[n-k-1]}{[n-k]}  \frac{[n-k]}{[n-k+1]} \VertexDefL\\
\end{align*}

We then apply \eqref{ref3} to obtain $X_{k+1}'$ and hence have proven the lemma by induction.

\end{proof}

Next, we use our elements $X_k$ to formulate the following theorem, which we will use to relate elements of $\mathcal{A}_{1,1}$ with the polynomials $F_{k,n}.$

\begin{thm} \label{MainRelations} For $2 \le k \le n-1$,
\begin{equation*}
    \AnnBraid=\VertexDef +q \ \VertexDefb + q^{-1}\ \RVertexDefb+ \LowVertexDef.
\end{equation*}    
 For the special cases of $k\in \{1,n\}$ we have
\begin{equation*}
    \OneAnnBraid=\OneVertexDef + q \ \OneVertexDefb + q^{-1} \ \OneRVertexDefb , 
\end{equation*}   

\begin{equation*}
    \NAnnBraid=q\ \NVertexDefb+q^{-1}\ \NRVertexDefb+\NLowVertexDef.
\end{equation*}   

Using the notation \ref{notation}, the equations can be written as
\begin{align}
l_1^*&=qT+q^{-1}T^{-1}+X_1, \label{InitialGTHM} \\
l_k^*&=qTX_{k-1}+q^{-1}T^{-1}X_{k-1}+X_{k}+X_{k-2}, \text{ for } 2 \le k \le n-1, \label{MidGTHM} \\
l_n^*&=qTX_{n-1}+q^{-1}T^{-1}X_{n-1}+X_{n-2}\label{EndGTHM}
\end{align}

\end{thm}

\begin{proof}
Equation \eqref{InitialGTHM} follows from Equation \eqref{1kbrading} when $k=1$.  \\
For the proof of Equation \eqref{MidGTHM}, first by Equation \eqref{1kbrading},
\begin{align*}
(-1)^k \AnnBraid =  \VertexDefa - \frac{q^{n-k+1}}{[n-k+1]} \VertexDefaBr -q^{-1} \RVertexDefaBr .
\end{align*}

Then by using the definition of $X_k$ from \eqref{VertDefInd},

\begin{align*}
 \VertexDefa = (-1)^k \left(\VertexDef +\frac{[n-k]}{[n-k+1]} \VertexDefb \right).
 \end{align*}

 By the definition of $X_{k-1}$ from Equation \eqref{VertDefInd}, 
 \begin{align*}
     \RVertexDefaBr= \VertexDefaIso = 
(-1)^{k-1} \left(\RVertexDefb + \frac{[n-k+1]}{[n-k+2]} \LowVertexDef \right).
 \end{align*}

By the expression for $X_{k-1}$ from Equation \eqref{Reflect}, 
\begin{align*}
    \VertexDefaBr=\RVertexDefaIso= (-1)^{k-1} \left(\VertexDefb + \frac{[n-k+1]}{[n-k+2]} \LowVertexDef \right).
\end{align*}

So \begin{align*}
\AnnBraid =  \quad & \VertexDef 
+\left(  \frac{[n-k]}{[n-k+1]} + \frac{q^{n-k+1}}{[n-k+1]}  \right) \VertexDefb \\ 
\quad & + q^{-1}\RVertexDefb  + \left( \frac{q^{n-k+1}}{[n-k+1]}+q^{-1} \right) \frac{[n-k+1]}{[n-k+2]}  \LowVertexDef, \\
\end{align*}

which is the desired expression since

\begin{align*}
\  \frac{[n-k]}{[n-k+1]}  + & \frac{q^{n-k+1}}{[n-k+1]} =q \ \text{ and } \ 
\left( \frac{q^{n-k+1}}{[n-k+1]}+q^{-1} \right) \frac{[n-k+1]}{[n-k+2]}=1. 
\end{align*}

For the proof of Equation \eqref{EndGTHM}, first recall that $\NBraida=0$ by convention. So by Equation \eqref{1kbrading} when $k=n$, 
\begin{align*}
(-1)^n \NAnnBraid =  - q \NVertexDefaBr -q^{-1} \NRVertexDefaBr .
\end{align*}

Then by Equations \eqref{VertDefInd} and \eqref{Reflect} for $X_{n-1}$, 
 \begin{align*}
    (-1)^n \NAnnBraid =
&(-1)^{n} q \left(\NRVertexDefb + \frac{1}{[2]} \NLowVertexDef \right) \\
&+(-1)^{n} q^{-1} \left(\NVertexDefb + \frac{1}{[2]} \NLowVertexDef \right).
\end{align*}

Thes gives us Equation \eqref{EndGTHM} as desired since $[2]= q+q^{-1}$.

\end{proof}

\subsection{The map $\varphi^*$}
We will construct an algebra map from a subring of $\mathbb{Z}[\Lambda]$ to $\mathcal{A}_{1,1}.$ The first step is to show certain elements of $\mathcal{A}_{1,1}$ commute with each other. Recall the elements from Notation \ref{notation}.

\begin{lemma}\label{commute}
The elements $T,T^{-1},$ $l_j^*$ for $1 \leq j \leq n,$ and $X_k$ for $1 \leq k \leq n-1$ are in the center of the algebra $\mathcal{A}_{1,1}$. 
\end{lemma}

\begin{proof}
We will prove this inductively. First, by equation (\ref{InitialGTHM}) we have

\begin{equation*}
l_1^*=qT+q^{-1}T^{-1}+X_1.
\end{equation*}

Since the elements $l_1^*, T, T^{-1}$ are seen to be in the center of $\mathcal{A}_{1,1}$ by considering simple isotopies, we deduce that $X_1$ is in the center as well. Since $X_0=\text{id}$ by convention, it is in the center as well.

For the inductive step, assume that $2 \leq k \leq n-1$ and that we have shown $X_l$ is in the center for $l< k.$ Then from equation (\ref{MidGTHM}) we have

\begin{equation*}
l_k^*=qTX_{k-1}+q^{-1}T^{-1}X_{k-1}+X_{k}+X_{k-2}.
\end{equation*}

By consideration of isotopy, $l_k^*$ is in the center of $\mathcal{A}_{1,1}.$ Our inductive assumption implies that $X_{k-1}$ and $X_{k-2}$ are in the center as well. We then deduce that $X_k$ is in the center for all $k \leq n-1.$

\end{proof}

\begin{thm}\label{phi star}

There exists a unique algebra homomorphism $\varphi^*: \mathbb{Z}[\lambda_n^{\pm 1},F_{1,n-1}, \dots, F_{n-1,n-1}] \rightarrow \mathcal{A}_{1,1}$ defined on generators in the following way

\begin{align*}
   \varphi^* ( F_{k,n-1} ) = \VertexDef =X_k, \quad 1\le k \le n-1, \\ 
      \varphi^*( \lambda_n  ) = q \ \OneVertexDefb =q T, \quad
   \varphi^*( \lambda_n^{-1} ) = q^{-1} \ \OneRVertexDefb =q^{-1} T^{-1} .
\end{align*}

Moreover, $\varphi^*$ satisfies the following identities

\begin{equation*}
\varphi^* ( F_{k,n}) = \AnnBraid =l_k^*, \quad 1\le k \le n.
\end{equation*}

Similarly, there exists a map $\varphi_*: \mathbb{Z}[\lambda_n^{\pm 1},F_{1,n-1}, \dots, F_{n-1,n-1}] \rightarrow \mathcal{A}_{1,1}$ obtained from the definition of $\varphi^*$ on generators by switching $q$ to $q^{-1}.$

\end{thm}

\begin{proof}

Note that the proposed definition of $\varphi^*$ satisfies $\varphi^*(\lambda_n^{-1})=(\varphi^*(\lambda_n))^{-1}.$ We deduce that $\varphi^*$ is a well-defined algebra homomorphism because the elements $\lambda_n^{\pm 1}, F_{1,n-1}, \dots, F_{n-1,n-1}$ are algebraically independent and  Lemma \ref{commute} implies their images commute.

We then check that $\varphi^*(F_{k,n})=l_k^*$ for each $1 \leq k \leq n.$

For $k=1,$ we use equation (\ref{InitialGTHM}), and then the fact that $\varphi^*(F_{1,n-1})=X_1$ by definition, followed by Lemma \ref{evalue relations} to see that

\begin{align*}
l_1^*&=qT+q^{-1}T^{-1}+X_1\\
&=\varphi^*(\lambda_n+\lambda_n^{-1}+F_{1,n-1})\\
&=\varphi^*(F_{1,n}),
\end{align*}
as claimed.

For $2 \leq k \leq n-1,$ we use equation (\ref{MidGTHM}), the definition of $\varphi^*,$ and Lemma \ref{evalue relations} to see that

\begin{align*}
l_k^*&=qTX_{k-1}+q^{-1}T^{-1}X_{k-1}+X_k+X_{k-2}\\
&= \varphi^*(\lambda_nF_{k-1,n-1} +\lambda_n^{-1}F_{k-1,n-1}+F_{k-2,n-1}+F_{k,n-1})\\
&=\varphi^*(F_{k,n}),
\end{align*} as claimed.

Finally, for $k=n$, we have that by equation (\ref{EndGTHM}), the definition of $\varphi^*,$ and Lemma \ref{evalue relations}

\begin{align*}
l_n^*&=qTX_{n-1}+q^{-1}T^{-1}X_{n-1}+X_{n-2}\\
&=\varphi^*(\lambda_nF_{n-1,n-1}+\lambda_n^{-1}F_{n-1,n-1}+F_{n-2,n-1})\\
&=\varphi^*(F_{n,n}),
\end{align*} as claimed.

\end{proof}

\def\kEdge{
\begin{tric}
\draw (0,-2)--(0,2) node[pos=0.65,left,scale=0.7,black]{$\klabel$};
\end{tric}
}

\def\SumTwistd{
\begin{tric}
\draw (-1,1)--(-1,2) 
       node[midway,left,scale=0.7,black]{$\klabel$}
       (-1,2)..controls(-2,2)and(-2,2.5)..(-2,3) 
        node[midway,left,scale=0.7,black]{$\klabel-1$}
       (-1,2)..controls(0,3)..(0,7);
\filldraw[black] (-2.25,3.25) circle (1pt) 
                 (-2.5,3.5) circle (1pt) (-2.75,3.75) circle (1pt); 
\draw 
(-3,4)..controls(-3.5,4)and(-3.5,4.5)..(-3.5,5)node[midway,left,scale=0.7,black]{$3$}
(-3.5,5)..controls(-4.5,5)and(-4.5,5.5)..(-4.5,6)node[midway,left,scale=0.7,black]{$2$}
(-4.5,6)..controls(-5.5,6)and(-5.5,6.5)..(-5.5,7)
(-3.5,5)..controls(-2.5,5.5)..(-2.5,7) 
(-4.5,6)..controls(-3.5,6)and(-3.5,6.5)..(-3.5,7);
\begin{scope}[yshift= 14cm]
\draw  (-1,-1)--(-1,-2) 
       node[midway,left,scale=0.7,black]{$\klabel$}
       (-1,-2)..controls(-2,-2)and(-2,-2.5)..(-2,-3) 
        node[midway,left,scale=0.7,black]{$\klabel-1$}
       (-1,-2)..controls(0,-3)..(0,-7)node[pos=0.3,right,scale=0.7,black]{$1$};
\filldraw[black] (-2.25,-3.25) circle (1pt) 
                 (-2.5,-3.5) circle (1pt) (-2.75,-3.75) circle (1pt)
      (-0.9,-6.5) circle (1pt) (-1.25,-6.5) circle (1pt) (-1.6,-6.5) circle (1pt); 
\draw 
(-3,-4)..controls(-3.5,-4)and(-3.5,-4.5)..(-3.5,-5)node[midway,left,scale=0.7,black]{$3$}
(-3.5,-5)..controls(-4.5,-5)and(-4.5,-5.5)..(-4.5,-6)node[midway,left,scale=0.7,black]{$2$}

(-4.5,-6)..controls(-5.5,-6)and(-5.5,-6.5)..(-5.5,-7)
(-5.4,-6.3)node[above,scale=0.7,black]{$1$}
(-3.5,-5)..controls(-2.5,-5.5)..(-2.5,-7) node[pos=0.35,right,scale=0.7,black]{$1$}
(-4.5,-6)..controls(-3.5,-6)and(-3.5,-6.5)..(-3.5,-7)
(-3.6,-6.3)node[above,scale=0.7,black]{$1$};
\end{scope}

\end{tric}
}

\def\kEdgeOverP{
\begin{tric}
\draw (-1.5,0)--(3.5,0)node[pos=0.75,above,scale=0.7,black]{$\PowerSum$};
\draw [double=darkblue,ultra thick,white,line width=3pt](0,-2)--(0,2) node[pos=0.75,left,scale=0.7,black]{$\klabel$};
\end{tric}
}

\def\kEdgeUnderP{
\begin{tric}
\draw (0,-2)--(0,2) node[pos=0.75,left,scale=0.7,black]{$\klabel$};
\draw [double=darkblue,ultra thick,white,line width=3pt](-1.5,0)--(3.5,0)node[pos=0.75,above,scale=0.7,black]{$\PowerSum$};
\end{tric}
}

\def\SumTwistdOverP{
\begin{tric}
\draw (-6.5,7)--(2.5,7)node[pos=0.9,above,scale=0.7,black]{$\PowerSum$};

\draw [double=darkblue,ultra thick,white,line width=3pt]
       (-1,2)..controls(-2,2)and(-2,2.5)..(-2,3) 
        node[midway,left,scale=0.7,black]{$\klabel-1$};
\filldraw[black] (-2.25,3.25) circle (1pt) 
                 (-2.5,3.5) circle (1pt) (-2.75,3.75) circle (1pt); 
\begin{scope}[yshift= 14cm]
\draw  [double=darkblue,ultra thick,white,line width=3pt](-1,-1)--(-1,-2) 
       node[midway,left,scale=0.7,black]{$\klabel$}
       (-1,-2)..controls(-2,-2)and(-2,-2.5)..(-2,-3) 
        node[midway,left,scale=0.7,black]{$\klabel-1$}
       (-1,-2)..controls(0,-3)..(0,-7)..controls(0,-11)..(-1,-12)
       (-0.05,-3.1)node[above,scale=0.7,black]{$1$};
\filldraw[black] (-2.25,-3.25) circle (1pt) 
                 (-2.5,-3.5) circle (1pt) (-2.75,-3.75) circle (1pt)
      (-0.9,-6.5) circle (1pt) (-1.25,-6.5) circle (1pt) (-1.6,-6.5) circle (1pt); 
      
\draw [double=darkblue,ultra thick,white,line width=3pt]
(-3,-4)..controls(-3.5,-4)and(-3.5,-4.5)..(-3.5,-5)node[midway,left,scale=0.7,black]{$3$}
(-3.5,-5)..controls(-4.5,-5)and(-4.5,-5.5)..(-4.5,-6)node[midway,left,scale=0.7,black]{$2$}

(-4.5,-6)..controls(-5.5,-6)and(-5.5,-6.5)..(-5.5,-7)..controls(-5.5,-7.5)and(-5.5,-8)..(-4.5,-8)
(-5.4,-6.3)node[above,scale=0.7,black]{$1$}
(-3.5,-5)..controls(-2.5,-5.5)..(-2.5,-7)..controls(-2.5,-8.5)..(-3.5,-9)
(-2.85,-5.4)node[right,scale=0.7,black]{$1$}
(-4.5,-6)..controls(-3.5,-6)and(-3.5,-6.5)..(-3.5,-7)..controls(-3.5,-7.5)and(-3.5,-8)..(-4.5,-8)
(-3.6,-6.3)node[above,scale=0.7,black]{$1$};

\draw 
(-3,-10)..controls(-3.5,-10)and(-3.5,-9.5)..(-3.5,-9)node[midway,left,scale=0.7,black]{$3$}
(-3.5,-9)..controls(-4.5,-9)and(-4.5,-8.5)..(-4.5,-8)node[midway,left,scale=0.7,black]{$2$};

\end{scope}

\draw(-1,1)--(-1,2) 
       node[midway,left,scale=0.7,black]{$\klabel$};
\end{tric}
}

\def\SumTwistdUnderP{
\begin{tric}
\draw (-1,1)--(-1,2) 
       node[midway,left,scale=0.7,black]{$\klabel$}
       (-1,2)..controls(-2,2)and(-2,2.5)..(-2,3) 
        node[midway,left,scale=0.7,black]{$\klabel-1$}
       (-1,2)..controls(0,3)..(0,7);
\filldraw[black] (-2.25,3.25) circle (1pt) 
                 (-2.5,3.5) circle (1pt) (-2.75,3.75) circle (1pt); 
\draw 
(-3,4)..controls(-3.5,4)and(-3.5,4.5)..(-3.5,5)node[midway,left,scale=0.7,black]{$3$}
(-3.5,5)..controls(-4.5,5)and(-4.5,5.5)..(-4.5,6)node[midway,left,scale=0.7,black]{$2$}
(-4.5,6)..controls(-5.5,6)and(-5.5,6.5)..(-5.5,7)
(-3.5,5)..controls(-2.5,5.5)..(-2.5,7) 
(-4.5,6)..controls(-3.5,6)and(-3.5,6.5)..(-3.5,7);
\begin{scope}[yshift= 14cm]
\draw  (-1,-1)--(-1,-2) 
       node[midway,left,scale=0.7,black]{$\klabel$}
       (-1,-2)..controls(-2,-2)and(-2,-2.5)..(-2,-3) 
        node[midway,left,scale=0.7,black]{$\klabel-1$}
       (-1,-2)..controls(0,-3)..(0,-7)node[pos=0.3,right,scale=0.7,black]{$1$};
\filldraw[black] (-2.25,-3.25) circle (1pt) 
                 (-2.5,-3.5) circle (1pt) (-2.75,-3.75) circle (1pt)
      (-0.9,-6.5) circle (1pt) (-1.25,-6.5) circle (1pt) (-1.6,-6.5) circle (1pt); 
\draw 
(-3,-4)..controls(-3.5,-4)and(-3.5,-4.5)..(-3.5,-5)node[midway,left,scale=0.7,black]{$3$}
(-3.5,-5)..controls(-4.5,-5)and(-4.5,-5.5)..(-4.5,-6)node[midway,left,scale=0.7,black]{$2$}

(-4.5,-6)..controls(-5.5,-6)and(-5.5,-6.5)..(-5.5,-7)
(-5.4,-6.3)node[above,scale=0.7,black]{$1$}
(-3.5,-5)..controls(-2.5,-5.5)..(-2.5,-7) node[pos=0.35,right,scale=0.7,black]{$1$}
(-4.5,-6)..controls(-3.5,-6)and(-3.5,-6.5)..(-3.5,-7)
(-3.6,-6.3)node[above,scale=0.7,black]{$1$};
\end{scope}

\draw [double=darkblue,ultra thick,white,line width=3pt](-6.5,7)--(2.5,7)node[pos=0.9,above,scale=0.7,black]{$\PowerSum$};

\end{tric}
}

\subsection{Constructing transparent elements}
In this section we prove the main result of this paper, which is that the polynomials $Q_{2n}^{(N,k)}$ can be used to construct transparent elements in $C_n$ skein modules and central elements in $C_n$ skein algebras of surfaces.

\begin{definition}
Given a framed knot $K$ in an oriented $3\text{-manifold } M$ and a polynomial
\begin{equation*}
P=\sum a_{i_1\cdots i_n}f_1^{i_1}\cdots f_n^{i_n} \in \mathcal{R}[f_1,\dots,f_n],    
\end{equation*} 
the result of threading the polynomial $P$ along the knot $K$ is denoted by $K^{[P]} \in \mathcal{S}_q^{C_n}(M)$ and is given by

\begin{equation*}
K^{[P]}=\sum a_{i_1\cdots i_n}K^{[f_1^{i_1}\cdots f_n^{i_n}]},
\end{equation*}
where $K^{[f_1^{i_1}\cdots f_n^{i_n}]}$ is given by taking $i_1+\cdots+i_n$ parallel copies of the knot $K$ in the framing direction and giving $i_k$ of these knots the label $k$, for each $k=1, \dots n.$ The resulting element does not depend on which collection is chosen for each label, since the copies are parallel and can be rearranged using Reidemeister 2 moves.
\end{definition}

\begin{definition}
We say that a polynomial $P \in \mathcal{R}[f_1,\dots,f_n]$ is transparent if the following equality holds for any oriented $3\text{-manifold } M$, for any web $W \in \mathcal{S}_q^{C_n}(M)$ and for any two framed knots $K_1,K_2$ disjoint from $W$ which are isotopic by an isotopy which is allowed to pass through $W$: $[K_1^{[P]} \cup W]=[K_2^{[P]}\cup W] \in \mathcal{S}_q^{C_n}(M).$
\end{definition}

Evidently, if $P$ is transparent, then if $\Sigma$ is a surface and $K$ is a framed knot in $\Sigma \times (0,1),$ the element $K^{[P]}$ will be in the center of $\mathcal{S}_q^{C_n}(\Sigma).$ We aim to show that $Q_{2n}^{(N,k)}$ is transparent when $q$ is an appropriate root of unity. To show that a knot $K^{[Q_{2n}^{(N,k)}]}$ can pass through a web $W$, it suffices to show that the element can pass through any edge of a web $W.$ The following lemma shows further that it suffices to show that $K^{[Q_{2n}^{(N,k)}]}$ can pass through an edge labeled by $1$, which relies on the fact that we have assumed our ground ring contains $[2]^{-1} ,\dots ,[n]^{-1}.$ 

\begin{lemma}\label{transparency1}
  Assume that $K^{[Q_{2n}^{(N,k)}]}$ can pass through an edge carrying the label 1. Then it can pass through an edge carrying any label $k \in \{1, \dots n\}.$
\end{lemma}

\begin{proof}
We first repeatedly apply relation (\ref{defskein}c) to the edge labeled by $k$ so as to create parallel edges all labeled by $1$.
  \[ \kEdge= \frac{1}{[\klabel]!}\SumTwistd \]
  We can then pass the element through each edge labeled by $1$ in the following way, where the horizontal strand depicts a piece of the knot threaded by $Q_{2n}^{(N,k)}$.
  \[\kEdgeOverP
  =\frac{1}{[\klabel]!}\SumTwistdOverP
  =\frac{1}{[\klabel]!}\SumTwistdUnderP
  =\kEdgeUnderP\]
\end{proof}

Recall that throughout the paper we assume that our specialization of $q$ satisfies $[2]^{-1},\dots,[n]^{-1} \in \mathcal{R}.$

\begin{thm}\label{main theorem}
When $q^{2N}=1,$ the polynomial $Q_{2n}^{(N,k)}$ is transparent for each $k=1, \dots n.$ Consequently, for a surface $\Sigma$, if $K$ is an element of  $\mathcal{S}_q^{C_n}(\Sigma)$ represented by a framed knot $K$, then for each $k=1,\dots, n$, the element obtained by threading the power sum polynomial $Q_{2n}^{(N,k)}(f_1,...,f_n)$ along $K$ is a central element of the skein algebra.
\end{thm}

\begin{proof}
By Lemma \ref{transparency1}, it suffices to show that the element can pass through an edge labeled by $1.$ By considering a tubular neighborhood of $K$ which is a thickened annulus, this amounts to showing the following equality in $\mathcal{A}_{1,1}$: $Q_{2n}^{(N,k)}(l_1,\dots, l_n)^*=Q_{2n}^{(N,k)}(l_1,\dots, l_n)_* \in \mathcal{A}_{1,1}$ when $q^{2N}=1.$ We will do this by using the algebra homomorphisms $(-)^*$ and $(-)_*$ from Definition \ref{star homomorphisms} and $\varphi^*,\varphi_*$ from Theorem \ref{phi star}.

We have that

\begin{align*}
Q_{2n}^{(N,k)}(l_1,\dots, l_n)^*&=Q_{2n}^{(N,k)}(l_1^*,\dots,l_n^*)\\
&=Q_{2n}^{(N,k)}(\varphi^*(F_{1,n}),\dots,\varphi^*(F_{n,n}))\\
&=\varphi^*(Q_{2n}^{(N,k)}(F_{1,n},\dots,F_{n,n}))\\
&=\varphi^*(F_{k,n}^{(N)}).
\end{align*}

Similarly, we have that

\begin{equation*}
Q_{2n}^{(N,k)}(l_1,\dots, l_n)_*=\varphi_*(F_{k,n}^{(N)}).
\end{equation*}

Recall from Lemma \ref{evalue relations} we have for any $1 \leq k \leq n$,

\begin{equation*}
F_{k,n}=(\lambda_n+\lambda_n^{-1})F_{k-1,n-1}+F_{k-2,n-1}+F_{k,n-1}.
\end{equation*}

Recall that $F_{k,n}^{(N)}$ is obtained from $F_{k,n}$ by replacing each variable $\lambda_{i}^{\pm 1}$ with $\lambda_{i}^{\pm N}.$ Consequently, we use Proposition \ref{Property of Type C polynomial} to show that

\begin{align*}
F_{k,n}^{(N)}=&(\lambda_n^N+\lambda_n^{-N})F_{k-1,n-1}^{(N)}+F_{k-2,n-1}^{(N)}+F_{k,n-1}^{(N)}\\
=&(\lambda_n^N+\lambda_n^{-N})Q_{2(n-1)}^{(N,k-1)}(F_{1,n-1},\dots,F_{n-1,n-1})+Q_{2(n-1)}^{(N,k-2)}(F_{1,n-1},\dots,F_{n-1,n-1})\\
&+Q_{2(n-1)}^{(N,k)}(F_{1,n-1},\dots,F_{n-1,n-1}).
\end{align*}

We then compute that

\begin{align*}
\varphi^*(F_{k,n}^{(N)})=&(\varphi^*(\lambda_n^N)+\varphi^*(\lambda_n^{-N}))Q_{2(n-1)}^{(N,k-1)}(\varphi^*(F_{1,n-1}),\dots,\varphi^*(F_{n-1,n-1}))+Q_{2(n-1)}^{(N,k-2)}(\varphi^*(F_{1,n-1}),\dots,\varphi^*(F_{n-1,n-1}))\\
&+Q_{2(n-1)}^{(N,k)}(\varphi^*(F_{1,n-1}),\dots,\varphi^*(F_{n-1,n-1}))\\
=&(q^NT^N-q^{-N}T^{-N})Q_{2(n-1)}^{(N,k-1)}(X_1,\dots,X_{n-1}) +Q_{2(n-1)}^{(N,k-2)}(X_1,\dots,X_{n-1})\\
&+Q_{2(n-1)}^{(N,k)}(X_1,\dots,X_{n-1}).
\end{align*}

We then use the fact that the definition of $\varphi_*$ is given by replacing $q$ by $q^{-1}$ in the definition of $\varphi^*$, along with the fact that the elements $X_i$ from Definition \ref{VertDef} are unchanged by the replacement $q\mapsto q^{-1}$ to compute that

\begin{align*}
\varphi_*(F_{k,n}^{(N)})=&(q^{-N}T^N-q^{N}T^{-N})Q_{2(n-1)}^{(N,k-1)}(X_1,\dots,X_{n-1}) +Q_{2(n-1)}^{(N,k-2)}(X_1,\dots,X_{n-1})\\
&+Q_{2(n-1)}^{(N,k)}(X_1,\dots,X_{n-1}).
\end{align*}

Finally, we deduce that

\begin{equation*}
\varphi^*(F_{k,n}^{(N)})-\varphi_*(F_{k,n}^{(N)})=[(q^{N}-q^{-N})T^N+(q^{-N}-q^{N})T^{-N}]Q_{2(n-1)}^{(N,k-1)}(X_{1},\dots,X_{n-1}).
\end{equation*}

Therefore, whenever $q$ satisfies $q^{2N}=1$ we have that $\varphi^*(F_{k,n}^{(N)})=\varphi_*(F_{k,n}^{(N)}),$ as required.

\end{proof}

\appendix

\section{Proof of the braiding formula} \label{Braidproof}

We explain the details of the proof for Equation \eqref{1kbrading} in this section. We need the following definition and proposition for the graphical calculations in the proof.

\def\FlowVert
{\begin{tric}
\draw [scale=2] (0,0)--(90:1) (0,0)--(210:1) (0,0)--(330:1);
\draw [scale=2] 
     (90:1)node[black,anchor=south,scale=0.7]{$k+l$}
      (210:1)node[black,anchor=north,scale=0.7]{$k$}
      (330:1)node[black,anchor=north,scale=0.7]{$l$}; 
\end{tric}
}

\def\FlowVertA
{\begin{tric} [scale=.3,anchorbase]
	\draw[very thick] (-1,0) to node[below=-1pt,scale=0.7]{$k{-}1$} (1,0);
	\draw[very thick] (-1,0) to node[left,scale=0.7]{$1$} (0,1.732);
	\draw[very thick] (1,0) to node[right,scale=0.7]{$k{+}l{-}1$} (0,1.732);
	\draw[very thick] (0,1.732) to (0,3.232) node[above,scale=0.7]{$k{+}l$};
	\draw[very thick] (-2.3,-.75) node[below,scale=0.7]{$k$} to (-1,0);
	\draw[very thick] (2.3,-.75) node[below,scale=0.7]{$l$} to (1,0);
\end{tric}
}

\def\FlowVertB
{\begin{tric} [scale=.3,anchorbase]
	\draw[very thick] (-1,0) to node[below=-1pt,scale=0.7]{$l{-}1$} (1,0);
	\draw[very thick] (-1,0) to node[left,scale=0.7]{$k{+}l{-}1$} (0,1.732);
	\draw[very thick] (1,0) to node[right,scale=0.7]{$1$} (0,1.732);
	\draw[very thick] (0,1.732) to (0,3.232) node[above,scale=0.7]{$k{+}l$};
	\draw[very thick] (-2.3,-.75) node[below,scale=0.7]{$k$} to (-1,0);
	\draw[very thick] (2.3,-.75) node[below,scale=0.7]{$l$} to (1,0);
\end{tric}
}

\begin{definition}[{\cite[Definition 3.2]{BERT21}}]
    The flow vertex in ${\text{Hom}}_{\textbf{Web}_q(\mathfrak{sp}_{2n})}(k \otimes l , k+l)$
    is defined inductively by the following relation.  
    \begin{equation}
        \FlowVert:=\frac{1}{[k]}\FlowVertA=\frac{1}{[l]}\FlowVertB. \label{FlowVert}
    \end{equation} 
\end{definition}

\begin{remark}
    Notice that when $k=1$ or $l=1$, Equation $\eqref{FlowVert}$ follows from $(\ref{defskein}\text{d})$ and $(\ref{defskein}\text{c})$. 
\end{remark}

\def\Flowassocia
{\begin{tric}
\draw[scale=0.8] (0,0)..controls(0,0.5)and(0.2,0.7)..(0.5,1)
      (1,0)..controls(1,0.5)and(0.8,0.7)..(0.5,1)
      (0.5,1)..controls(0.5,1.5)and(0.3,1.7)..(0,2) 
              node[right,black,midway, scale=0.7]{$l+m$}
      (-1,0)..controls(-1,1)and(-0.6,1.5)..(0,2)
      (0,2)--(0,3)
      (0,0)node[below,black,scale=0.7]{$l$}
      (1,0)node[below,black,scale=0.7]{$m$}
      (-1,0)node[below,black,scale=0.7]{$k$}
      (0,3)node[above,black,scale=0.7]{$k+l+m$};
\end{tric}
}

\def\Flowassoci
{\begin{tric}
\draw[scale=0.8] (0,0)..controls(0,0.5)and(-0.2,0.7)..(-0.5,1)
      (-1,0)..controls(-1,0.5)and(-0.8,0.7)..(-0.5,1)
      (-0.5,1)..controls(-0.5,1.5)and(-0.3,1.7)..(0,2) 
              node[left,black,midway, scale=0.7]{$k+l$}
      (1,0)..controls(1,1)and(0.6,1.5)..(0,2)
      (0,2)--(0,3)
      (0,0)node[below,black,scale=0.7]{$l$}
      (-1,0)node[below,black,scale=0.7]{$k$}
      (1,0)node[below,black,scale=0.7]{$m$}
      (0,3)node[above,black,scale=0.7]{$k+l+m$};
\end{tric}
}

\begin{proposition}[{\cite[Theorem 3.6]{BERT21}}]
    \begin{align}
        \Flowassoci=\Flowassocia \label{FlowAssoci}\\
        \FlowSkeinbigon = 0 \label{VanishBigon}
    \end{align}
\end{proposition}

\begin{proof}[Proof of Theorem \ref{Braiding theorem}]
Now we prove Equation \eqref{1kbrading} by induction. First, by \eqref{InitialBraid}, we know that \eqref{1kbrading} holds when $k=1$. 
Now suppose that Equation \eqref{1kbrading} holds for k, we prove that \eqref{1kbrading} holds for $k+1$ by the following calculations.

\def\APPBraidaa
{\begin{tric}
\draw  (1.5,1.5)--(-0.7,0.8)--(-1.5,1.5)
(-0.7,0.8)--(-0.7,0)node[midway,right,black,scale=0.7]{$k$}
(-1.4,-1)--(-1.5,-1.7)
(0,-1)--(1.5,-1.7);

\draw (-0.7,0)--(0,-1) node[pos=0.4,right,black,scale=0.7]{$1$}
      (0,-1)--(-1.4,-1) node[midway,below,black,scale=0.7]{$k$}
      (-1.4,-1)--(-0.7,0) node[pos=0.6,left,black,scale=0.7]{$k+1$}; 

\draw (-1.5,-1.7) node[below,black,scale=0.7]{$1$};
\draw (1.5,1.5) node[above,black,scale=0.7]{$1$};
\draw (1.5,-1.7) node[below,black,scale=0.7]{$k+1$};
\draw (-1.5,1.5) node[above,black,scale=0.7]{$k+1$};
\end{tric}}

\def\APPBraidab
{\begin{tric}
\draw  (1.5,1.5)--(-0.7,0.8)--(-1.5,1.5)
(-0.7,0.8)--(-0.7,0)node[midway,right,black,scale=0.7]{$k$}
(-1.4,-1)--(-1.5,-1.7)
(0,-1)--(1.5,-1.7);

\draw (-0.7,0)--(0,-1) node[pos=0.4,right,black,scale=0.7]{$1$}
      (0,-1)--(-1.4,-1) node[midway,below,black,scale=0.7]{$k$}
      (-1.4,-1)--(-0.7,0) node[pos=0.6,left,black,scale=0.7]{$k-1$}; 

\draw (-1.5,-1.7) node[below,black,scale=0.7]{$1$};
\draw (1.5,1.5) node[above,black,scale=0.7]{$1$};
\draw (1.5,-1.7) node[below,black,scale=0.7]{$k+1$};
\draw (-1.5,1.5) node[above,black,scale=0.7]{$k+1$};
\end{tric}}

\def\APPBraidac
{\begin{tric}
\draw  (0.7,-0.7)--(1.5,-1.5) (0.7,-0.7)--(1.5,1.5) (-0.7,0.7)--(-1.5,1.5);

\draw(-0.7,0.7)..controls(-1.2,0.2)..(-0.7,-0.3) node[midway,left,black,scale=0.7]{$k$}
(-0.7,0.7)..controls(-0.2,0.2)..(-0.7,-0.3)
node[midway,right,black,scale=0.7]{$1$};

\draw (-0.7,-0.3)--(-0.7,-1)node[midway,left,black,scale=0.7]{$k+1$}
      (-0.7,-1)--(-1.5,-1.5)
      (-0.7,-1)--(0.7,-0.7) node[midway,below,black,scale=0.7]{$k$};

\draw (-1.5,-1.5) node[below,black,scale=0.7]{$1$};
\draw (1.5,1.5) node[above,black,scale=0.7]{$1$};
\draw (1.5,-1.5) node[below,black,scale=0.7]{$k+1$};
\draw (-1.5,1.5) node[above,black,scale=0.7]{$k+1$};
\end{tric}}

\def\APPBraidad
{\begin{tric}
\draw  (1,-0.7)--(1.5,-1.5) (1,-0.7)--(1.5,1.5) (-0.7,0.7)--(-1.5,1.5);

\draw(-0.7,0.7)--(-1.2,-0.5) node[pos=0.4,left,black,scale=0.7]{$k$}
(-0.7,0.7)--(0.2,-0.5)
node[pos=0.35,right,black,scale=0.7]{$1$};

\draw (-1.2,-0.5)--(0.2,-0.5) node[midway,below,black,scale=0.7]{$k-1$}
      (-1.2,-0.5)--(-1.5,-1.5)
      (0.2,-0.5)--(1,-0.7) node[midway,above,black,scale=0.7]{$k$};

\draw (-1.5,-1.5) node[below,black,scale=0.7]{$1$};
\draw (1.5,1.5) node[above,black,scale=0.7]{$1$};
\draw (1.5,-1.5) node[below,black,scale=0.7]{$k+1$};
\draw (-1.5,1.5) node[above,black,scale=0.7]{$k+1$};
\end{tric}}

\def\APPBraidae
{\begin{tric}
\draw (-1.5,-1.5)--(-0.3,-1) (1.5,-1.5)--(0.7,-1) (-1.5,1.5)--(-0.7,1) (1.5,1.5)--(0.3,1);

\draw(-0.7,1)--(-0.3,-0.2)node[midway,left,black,scale=0.7]{$k$}
(-0.3,-0.2)--(-0.3,-1)node[midway,left,black,scale=0.7]{$k+1$};

\draw(0.7,-1)--(0.3,0.2)node[midway,right,black,scale=0.7]{$1$}
(0.3,0.2)--(0.3,1)node[midway,right,black,scale=0.7]{$2$};

\draw(0.3,0.2)--(-0.3,-0.2)node[pos=0.4,below,black,scale=0.7]{$1$}
     (-0.3,-1)--(0.7,-1)node[midway,below,black,scale=0.7]{$k$}
     (0.3,1)--(-0.7,1)node[midway,above,black,scale=0.7]{$1$}; 
     
\draw (-1.5,-1.5) node[below,black,scale=0.7]{$1$};
\draw (1.5,1.5) node[above,black,scale=0.7]{$1$};
\draw (1.5,-1.5) node[below,black,scale=0.7]{$k+1$};
\draw (-1.5,1.5) node[above,black,scale=0.7]{$k+1$};
\end{tric}}

\def\APPBraidaeCalcB
{\begin{tric}
\draw (-2,-1.5)--(-1.5,-0.7) (1.5,-1.5)--(0.7,-1) (-1.5,1.5)--(-0.7,1) (1.5,1.5)--(0.3,1);

\draw(-0.7,1)--(-1.5,-0.7)node[midway,left,black,scale=0.7]{$k$}
;

\draw(-0.2,-0.7)--(-1.5,-0.7)node[midway,below,black,scale=0.7]{$k-1$}
;

\draw(0.7,-1)--(0.3,0.2)node[midway,right,black,scale=0.7]{$1$}
(0.3,0.2)--(0.3,1)node[midway,right,black,scale=0.7]{$2$};

\draw(0.3,0.2)--(-0.2,-0.7)node[pos=0.4,left,black,scale=0.7]{$1$}
     (-0.2,-0.7)--(0.7,-1)node[midway,below,black,scale=0.7]{$k$}
     (0.3,1)--(-0.7,1)node[midway,above,black,scale=0.7]{$1$}; 
     
\draw (-2,-1.5) node[below,black,scale=0.7]{$1$};
\draw (1.5,1.5) node[above,black,scale=0.7]{$1$};
\draw (1.5,-1.5) node[below,black,scale=0.7]{$k+1$};
\draw (-1.5,1.5) node[above,black,scale=0.7]{$k+1$};
\end{tric}}

\def\APPBraidaeCalcA
{\begin{tric}
\draw (-2,-1.5)--(-1.5,-0.7) (1.5,-1.5)--(0.7,-1) (-1.5,1.5)--(-0.7,1) (1.5,1.5)--(0.3,1);

\draw(-0.7,1)--(-1.5,-0.7)node[midway,left,black,scale=0.7]{$k$}
;

\draw(-0.2,-0.7)--(-1.5,-0.7)node[midway,below,black,scale=0.7]{$k+1$}
;

\draw(0.7,-1)--(0.3,0.2)node[midway,right,black,scale=0.7]{$1$}
(0.3,0.2)--(0.3,1)node[midway,right,black,scale=0.7]{$2$};

\draw(0.3,0.2)--(-0.2,-0.7)node[pos=0.4,left,black,scale=0.7]{$1$}
     (-0.2,-0.7)--(0.7,-1)node[midway,below,black,scale=0.7]{$k$}
     (0.3,1)--(-0.7,1)node[midway,above,black,scale=0.7]{$1$}; 
     
\draw (-2,-1.5) node[below,black,scale=0.7]{$1$};
\draw (1.5,1.5) node[above,black,scale=0.7]{$1$};
\draw (1.5,-1.5) node[below,black,scale=0.7]{$k+1$};
\draw (-1.5,1.5) node[above,black,scale=0.7]{$k+1$};
\end{tric}}

\def\APPBraidaeCalcAA
{\begin{tric}
\draw (-2,-1.5)--(-1.5,-0.7) (1.8,-1.5)--(1.2,-1) (-1.5,1.5)--(-0.7,1) (1.5,1.5)--(1.2,-1);

\draw(-0.7,1)--(-1.5,-0.7)node[midway,left,black,scale=0.7]{$k$}
;

\draw(-0.2,-0.7)--(-1.5,-0.7)node[midway,below,black,scale=0.7]{$k+1$}; 

\draw(-0.7,1)--(-0.2,-0.7)
      (-0.3,0)node[black,left,scale=0.7]{$1$};

\draw (-0.2,-0.7)--(1.2,-1)node[midway,above,black,scale=0.7]{$k+2$}; 
     
\draw (-2,-1.5) node[below,black,scale=0.7]{$1$};
\draw (1.5,1.5) node[above,black,scale=0.7]{$1$};
\draw (1.8,-1.5) node[below,black,scale=0.7]{$k+1$};
\draw (-1.5,1.5) node[above,black,scale=0.7]{$k+1$};
\end{tric}}

\def\APPBraidaeCalcAB
{\begin{tric}
\draw (-2,-1.5)--(-1.5,-0.7) (1.8,-1.5)--(1.2,-1) (-1.5,1.5)--(-0.7,1) (1.5,1.5)--(1.2,-1);

\draw(-0.7,1)--(-1.5,-0.7)node[midway,left,black,scale=0.7]{$k$}
;

\draw(-0.2,-0.7)--(-1.5,-0.7)node[midway,below,black,scale=0.7]{$k+1$}; 

\draw(-0.7,1)--(-0.2,-0.7)
      (-0.3,0)node[black,left,scale=0.7]{$1$};

\draw (-0.2,-0.7)--(1.2,-1)node[midway,above,black,scale=0.7]{$k$}; 
     
\draw (-2,-1.5) node[below,black,scale=0.7]{$1$};
\draw (1.5,1.5) node[above,black,scale=0.7]{$1$};
\draw (1.8,-1.5) node[below,black,scale=0.7]{$k+1$};
\draw (-1.5,1.5) node[above,black,scale=0.7]{$k+1$};
\end{tric}}

\def\APPBraidaf
{\begin{tric}
\draw (-1.5,-1.5)--(-0.3,-1) (1.5,-1.5)--(0.7,-1) (-1.5,1.5)--(-0.7,1) (1.5,1.5)--(0.3,1);

\draw(-0.7,1)--(-0.3,-0.2)node[midway,left,black,scale=0.7]{$k$}
(-0.3,-0.2)--(-0.3,-1)node[midway,left,black,scale=0.7]{$k-1$};

\draw(0.7,-1)--(0.3,0.2)node[midway,right,black,scale=0.7]{$1$}
(0.3,0.2)--(0.3,1)node[midway,right,black,scale=0.7]{$2$};

\draw(0.3,0.2)--(-0.3,-0.2)node[pos=0.4,below,black,scale=0.7]{$1$}
     (-0.3,-1)--(0.7,-1)node[midway,below,black,scale=0.7]{$k$}
     (0.3,1)--(-0.7,1)node[midway,above,black,scale=0.7]{$1$}; 
     
\draw (-1.5,-1.5) node[below,black,scale=0.7]{$1$};
\draw (1.5,1.5) node[above,black,scale=0.7]{$1$};
\draw (1.5,-1.5) node[below,black,scale=0.7]{$k+1$};
\draw (-1.5,1.5) node[above,black,scale=0.7]{$k+1$};
\end{tric}}

\def\APPBraidafA
{\begin{tric}
\draw (-1.5,-1.5)--(-0.3,-1) (1.5,-1.5)--(0.7,-1) (-1.5,1.5)--(-0.7,1) (1.8,1.5)--(1,0.5);

\draw(-0.7,1)--(-0.3,-0.2)node[midway,left,black,scale=0.7]{$k$}
(-0.3,-0.2)--(-0.3,-1)node[midway,left,black,scale=0.7]{$k-1$};

\draw(0.7,-1)--(1,0.5)node[midway,right,black,scale=0.7]{$1$};

\draw(0.2,0.7)--(-0.3,-0.2)node[pos=0.6,right,black,scale=0.7]{$1$}
     (-0.3,-1)--(0.7,-1)node[midway,below,black,scale=0.7]{$k$}
     (0.2,0.7)--(-0.7,1)node[midway,above,black,scale=0.7]{$1$}; 

\draw     (0.2,0.7)--(1,0.5) node[pos=0.6,above,black,scale=0.7]{$2$};
     
\draw (-1.5,-1.5) node[below,black,scale=0.7]{$1$};
\draw (1.8,1.5) node[above,black,scale=0.7]{$1$};
\draw (1.5,-1.5) node[below,black,scale=0.7]{$k+1$};
\draw (-1.5,1.5) node[above,black,scale=0.7]{$k+1$};
\end{tric}}

\def\APPBraidafB
{\begin{tric}
\draw (-1.5,-1.5)--(-0.3,-1) (1.5,-1.5)--(0.7,-1) (-1.5,1.5)--(-0.7,1) (1.5,1.5)--(-0.7,1);

\draw(-0.7,1)--(-0.3,-0.2)node[midway,left,black,scale=0.7]{$k$}
(-0.3,-0.2)--(-0.3,-1)node[midway,left,black,scale=0.7]{$k-1$};

\draw(-0.3,-0.2)--(0.7,-1) node[pos=0.5,above,black,scale=0.7]{$1$}
     (-0.3,-1)--(0.7,-1)node[midway,below,black,scale=0.7]{$k$}; 
     
\draw (-1.5,-1.5) node[below,black,scale=0.7]{$1$};
\draw (1.5,1.5) node[above,black,scale=0.7]{$1$};
\draw (1.5,-1.5) node[below,black,scale=0.7]{$k+1$};
\draw (-1.5,1.5) node[above,black,scale=0.7]{$k+1$};
\end{tric}}

\def\APPBraidafC
{\begin{tric}
\draw (-1.5,-1.5)--(-0.3,-1) (1.5,-1.5)--(0.7,-1) (-1.5,1.5)--(-0.5,0.8) (1.8,1.5)--(1,0.5);

\draw
(-0.5,0.8)--(-0.3,-1)node[midway,left,black,scale=0.7]{$k-1$};

\draw(0.7,-1)--(1,0.5)node[midway,right,black,scale=0.7]{$1$};

\draw
     (-0.3,-1)--(0.7,-1)node[midway,below,black,scale=0.7]{$k$}; 

\draw     (-0.5,0.8)--(1,0.5) node[pos=0.6,above,black,scale=0.7]{$2$};
     
\draw (-1.5,-1.5) node[below,black,scale=0.7]{$1$};
\draw (1.8,1.5) node[above,black,scale=0.7]{$1$};
\draw (1.5,-1.5) node[below,black,scale=0.7]{$k+1$};
\draw (-1.5,1.5) node[above,black,scale=0.7]{$k+1$};
\end{tric}}

\def\APPBraidafD
{\begin{tric}
\draw (-1.5,-1.5)--(-0.3,-1) (1.5,-1.5)--(0.7,-1) (-1.5,1.5)--(-0.5,0.8) (1.8,1.5)--(1,0.5);

\draw
(-0.5,0.8)--(-0.3,-1)node[midway,left,black,scale=0.7]{$k-1$};

\draw(0.7,-1)--(1,0.5)node[midway,right,black,scale=0.7]{$1$};

\draw
     (-0.3,-1)--(0.7,-1)node[midway,below,black,scale=0.7]{$k$}; 

\draw     (-0.5,0.8)--(1,0.5) node[pos=0.6,above,black,scale=0.7]{$2$};
     
\draw (-1.5,-1.5) node[below,black,scale=0.7]{$1$};
\draw (1.8,1.5) node[above,black,scale=0.7]{$1$};
\draw (1.5,-1.5) node[below,black,scale=0.7]{$k+1$};
\draw (-1.5,1.5) node[above,black,scale=0.7]{$k+1$};
\end{tric}}

\begin{align} \label{PreSimplify}
   &  (-1)^k [k+1] \BraidP \stackrel{\eqref{klbraiding}}{=}  (-1)^k  \BraidPb \\
     \stackrel{ \substack{\eqref{InitialBraid} \\ \eqref{1kbrading} \\ \eqref{VanishBigon} } }{=} 
     &\frac{q^n}{[n]} \ \APPBraidaa - \frac{q^n}{[n]} \cdot \frac{q^{n-k+1}}{[n-k+1]} \ \APPBraidab  + q^{-1} \APPBraidac -q^{-2} \APPBraidad \notag \\
    &  + q^{-1} \APPBraidaeCalcB  +\frac{q^{n-k+1}}{[n-k+1]}\APPBraidaf - \APPBraidae  \notag.
\end{align}
We reduce each of the 7 diagrams appearing on the right side of \eqref{PreSimplify} to a linear combination of webs in 
\[ \left\{  \IndBraida,\IndBraidb,\IndBraidc\right\}.\] 
The first 5 diagrams reduce as follows
\begin{align*}
&\APPBraidaa  \stackrel{ \substack{\eqref{deducedskein} \\(\ref{defskein}\text{c}) \\ \eqref{FlowVert} \\ \eqref{VanishBigon} }}{=} 
\frac{[n+1]}{[n-k+1]}\IndBraidc, \quad   
\APPBraidab \stackrel{\eqref{FlowVert}}{=}
[k] \IndBraidc,  \\
&\APPBraidac \stackrel{(\ref{defskein}\text{c})}{=}[k+1] \IndBraidb,  \quad 
 \APPBraidad \stackrel{\eqref{FlowVert}}{=}[k]  \IndBraidb, \\
& \text{and} \quad \APPBraidaeCalcB
 \stackrel{ \substack{ \eqref{FlowVert} \\ \eqref{FlowAssoci} } }{=}  [2][k] \IndBraidb.\\
\end{align*}

The sixth diagram reduces as follows.

\begin{align} \label{SixthDiagram}
&\APPBraidaf
\stackrel{ \substack{ (\ref{defskein}\text{e}) \\ \eqref{VanishBigon} } }{=}  
\APPBraidafA -\frac{[n-1]}{[n]} \APPBraidafB 
\stackrel{  \eqref{FlowVert} }{=}
[2]\APPBraidafC -\frac{[n-1]}{[n]}[k]\IndBraidc  \\
 &\stackrel{ \eqref{FlowAssoci}  }{=}  
 [2]\APPBraidafB -\frac{[n-1]}{[n]}[k]\IndBraidc  \stackrel{  \eqref{FlowVert} }{=}
 \frac{[n+1]}{[n]}[k]\IndBraidc. \notag
\end{align}

To reduce the final diagram in \eqref{PreSimplify}, we begin by writing

\begin{equation}\label{final diagram}
\APPBraidae  \stackrel{ \eqref{deducedskein} }{=}
\APPBraidaeCalcA -\frac{[n-k]}{[n-k+1]} \APPBraidaf +\frac{[n-k]}{[n-k+1]} \APPBraidaeCalcB.
\end{equation}

We have already shown that the second and third diagrams in \eqref{final diagram} reduce in the following way.

\begin{align*}
&\APPBraidaf \stackrel{ \eqref{SixthDiagram} }{=} 
 \frac{[n+1]}{[n]}[k]\IndBraidc \text{ and } \APPBraidaeCalcB
 \stackrel{ \substack{ \eqref{FlowVert} \\ \eqref{FlowAssoci} } }{=}  [2][k] \IndBraidb.
\end{align*}

We reduce the first term in \eqref{final diagram} in the following way.

\begin{align*}
&\APPBraidaeCalcA \stackrel{ (\ref{defskein}\text{e}) }{=} 
\APPBraidaeCalcAA + \frac{[n-k-1]}{[n-k]}\APPBraidaeCalcAB - \frac{[n-k-1]}{[n]}\IndBraidc \\
&\stackrel{ \substack{ \eqref{FlowVert} \\ \eqref{deducedskein} \\(\ref{defskein}\text{c}) \\ \eqref{VanishBigon} } }{=} 
[k+1]\IndBraidbH+ \frac{[n-k-1]}{[n-k]}\frac{[n+1]}{[n-k+1]} \IndBraidb - \frac{[n-k-1]}{[n]}\IndBraidc \\
&\stackrel{\eqref{deducedskein}}{=} 
 [k+1] \IndBraida + \frac{[k+1][n-k-1]}{[n-k]} \IndBraidc 
     - \frac{[k+1][n-k-1]}{[n-k]} \IndBraidb \\
& \qquad  + \frac{[n-k-1]}{[n-k]}\frac{[n+1]}{[n-k+1]} \IndBraidb  - \frac{[n-k-1]}{[n]} \IndBraidc.
\end{align*}

We then replace each diagram in \eqref{PreSimplify} with the corresponding reduced form above to obtain

\begin{align*}    
    & (-1)^k [k+1] \BraidP \\
    = & \frac{q^n[n+1]}{[n][n-k+1]} \IndBraidc 
    -\frac{q^{2n-k+1}[k]}{[n][n-k+1]}\IndBraidc 
    + q^{-1}[k+1] \IndBraidb\\
     &  - q^{-2}[k]  \IndBraidb 
     + q^{-1}[2][k] \IndBraidb
     + \frac{q^{n-k+1}}{[n-k+1]}\frac{[n+1]}{[n]}[k]\IndBraidc\\
      &- [k+1] \IndBraida - \frac{[k+1][n-k-1]}{[n-k]} \IndBraidc 
     + \frac{[k+1][n-k-1]}{[n-k]} \IndBraidb \\
     & - \frac{[n-k-1]}{[n-k]}\frac{[n+1]}{[n-k+1]} \IndBraidb 
     + \frac{[n-k-1]}{[n]} \IndBraidc  \\ &
     + \frac{[n-k]}{[n-k+1]}\frac{[n+1]}{[n]}[k]\IndBraidc
      - \frac{[n-k]}{[n-k+1]}[2][k] \IndBraidb. 
\end{align*}
Finally, we collect coefficients and then verify the following identitis by Sage and hence finish the proof by induction.

\begin{align*}
    \frac{q^n[n+1]}{[n][n-k+1]}
        -\frac{q^{2n-k+1}[k]}{[n][n-k+1]}  
        + \frac{q^{n-k+1}}{[n-k+1]}\frac{[n+1]}{[n]}[k]
    - \frac{[k+1][n-k-1]}{[n-k]} 
     & \\
     + \frac{[n-k-1]}{[n]} + \frac{[n-k]}{[n-k+1]}\frac{[n+1]}{[n]}[k]   
   & = -[k+1]\left(-\frac{q^{n-k}}{[n-k]}\right) ,\\
q^{-1}[k+1] 
   - q^{-2}[k]   + q^{-1}[2][k]
   + \frac{[k+1][n-k-1]}{[n-k]} 
 & \\
    - \frac{[n-k-1]}{[n-k]}\frac{[n+1]}{[n-k+1]}
    - \frac{[n-k]}{[n-k+1]}[2][k]  
    & = -[k+1](-q^{-1}). 
\end{align*}

\end{proof}

\bibliographystyle{amsalpha}
\bibliography{references}
\end{document}